\documentclass[a4paper,10pt,reqno]{amsart}
\usepackage{amsfonts,amsmath,amssymb,amsthm}
\usepackage{graphicx}
\usepackage{fixmath}
\usepackage{hyperref}
\usepackage{tikz}
\usepackage[noadjust]{cite}
\newtheorem{theo}{Theorem}[section]

\newtheorem{lemma}[theo]{Lemma}

\newtheorem{rem}{Remark}

\newcommand{\old}[1]{{}}

\DeclareMathOperator{\Z}{Z}

\DeclareMathOperator{\W}{W}

\def\emptyset{\varnothing}

\usepackage{array}
\newcolumntype{L}[1]{>{\raggedright\let\newline\\\arraybackslash\hspace{0pt}}b{#1}}
\newcolumntype{C}[1]{>{\centering\let\newline\\\arraybackslash\hspace{0pt}}b{#1}}
\newcolumntype{R}[1]{>{\raggedleft\let\newline\\\arraybackslash\hspace{0pt}}m{#1}}

\begin{document}

\pagestyle{plain}

\title{Subtrees and independent subsets in unicyclic graphs and unicyclic graphs with fixed segment sequence}

\author{Eric Ould Dadah Andriantiana}
\address{Eric Ould Dadah Andriantiana\\
Department of Mathematics (Pure and Applied)\\
Rhodes University, PO Box 94\\
6140 Grahamstown\\
South Africa}
\email{E.Andriantiana@ru.ac.za}

\author{Hua Wang}
\address{Hua Wang\\
Department of Mathematical Sciences \\
Georgia Southern University \\
Statesboro, GA 30460, USA
}
\email{hwang@georgiasouthern.edu}
\thanks{This work was supported by grants from the Simons Foundation (\#245307) and the National Research Foundation of South Africa (grant 96310).}


\begin{abstract}
In the study of topological indices two negative correlations are well known: that between the number of subtrees and the Wiener index (sum of distances), and that between the Merrifield-Simmons index (number of independent vertex subsets) and the Hosoya index (number of independent edge subsets). That is, among a certain class of graphs, the extremal graphs that maximize one index usually minimize the other, and vice versa. In this paper, we first study the numbers of subtrees in unicyclic graphs and unicyclic graphs with a given girth, further confirming its opposite behavior to the Wiener index by comparing with known results.  We then consider the unicyclic graphs with a given segment sequence and characterize the extremal structure with the maximum number of subtrees. Furthermore, we show that these graphs are not extremal with respect to the Wiener index. We also identify the extremal structures that maximize the number of independent vertex subsets among unicyclic graphs with a given segment sequence, and show that they are not extremal with respect to the number of independent edge subsets. These results may be the first examples where the above negative correlation failed in the extremal structures between these two pairs of indices.
\end{abstract}

\maketitle

\section{Introduction}
The study of topological indices and the extremal graphs that maximize or minimize them have attracted much attention in the recent years. In particular, the {\it number of subtrees} $n(G)$ of a graph $G$ has been found to be ``negatively correlated'' to the {\it Wiener index} \cite{wiener1947structural}
$$
\W(G)=\sum_{u,v\in V(G)}d_G(u,v) 
$$
where $d_G(u,v)$ denotes the distance between $u$ and $v$ in $G$. This negative correlation was observed as among certain categories of graphs, the extremal structures that maximize $n(G)$ also minimizes $\W(G)$ and vice versa. In \cite{wagner2007} the correlations between a number of pairs of graph invariants were examined, $n(G)$ and $\W(G)$ were found to be the most (negatively) correlated. Along this line extremal graphs with respect to $n(G)$ \cite{zhang12,zhang2015,sze05,DadahAndriantiana2013} have been found to coincide with those with respect to $\W(G)$ \cite{schmuck2012greedy,zhang2015,wang2008trees,zhang2008wiener}
in many different classes of graphs.

Another pair of such ``negatively correlated'' graph invairants are the {\it Merrifield-Simmons index} \cite{merrifield1989topological}
$$
\sigma(G)=|\{B\subseteq V(G):\{\{u,v\}:u,v\in B\}\cap E(G)=\emptyset\}|
$$
and the {\it Hosoya index} \cite{hosoya1971topological}  
$$
\Z(G)=|\{B\subseteq E(G): e\cap e'=\emptyset\text{ for any } e\neq e'\text{ in } B\}|.
$$
They count the number of independent vertex subsets and the number of independent 
edge subsets, respectively. The extremal graphs with respect to $\Z(\cdot)$ and $\sigma(\cdot)$ are shown to coincide 
in several classes of graphs (that is, what maximizes one minimizes the other). For details we recommend the survey \cite{Wagner2010323} and the references therein.

A unicyclic graph is a connected graph that contains exactly one cycle. As many of the earlier extremal problems for these graph invariants were considered for various classes of trees, examining extremal problems for unicyclic graphs seems to be the natural next step. 
In the case of the Wiener index, some of such results have already been reported \cite{du2010minimum,Tan20161,Tan20151,Ren2013768}. 

We denote by $C_n,S_n$ and $P_n$ the $n$-vertex cycle, star and path, respectively.  For $n\geq 3$, let $US_n$ be the unicyclic graph obtained by adding an edge to join two leaves of $S_n$, and $UP_n$ be the unicyclic graph obtained by identifying a vertex of $C_3$ and an end of $P_{n-2}$ (Figure~\ref{Fig:US_UP}). We will first characterize the extremal structures among all unicyclic graphs that maximize or minimize the number of subtrees.

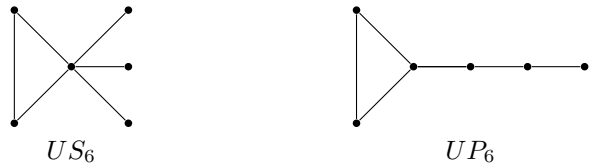
\begin{figure}[!ht]
\centering
    \begin{tikzpicture}[scale=0.75]
    \node[fill=black,circle,inner sep=1pt] (t1) at (0,0) {};
    \node[fill=black,circle,inner sep=1pt] (t2) at (-1,-1) {};
    \node[fill=black,circle,inner sep=1pt] (t3) at (-1,1) {};
    \node[fill=black,circle,inner sep=1pt] (t4) at (1,1) {};
    \node[fill=black,circle,inner sep=1pt] (t5) at (1,0) {};
    \node[fill=black,circle,inner sep=1pt] (t6) at (1,-1) {};
    \node at (0,-1.5) {$US_6$};
    \draw (t6)--(t1)--(t2)--(t3)--(t1)--(t4);
    \draw (t1)--(t5);
    \node[fill=black,circle,inner sep=1pt] (t11) at (6+0,0) {};
    \node[fill=black,circle,inner sep=1pt] (t12) at (6+-1,-1) {};
    \node[fill=black,circle,inner sep=1pt] (t13) at (6+-1,1) {};
    \node[fill=black,circle,inner sep=1pt] (t14) at (6+2,0) {};
    \node[fill=black,circle,inner sep=1pt] (t15) at (6+1,0) {};
    \node[fill=black,circle,inner sep=1pt] (t16) at (6+3,0) {};
    \node at (7,-1.5) {$UP_6$};
    \draw (t16)--(t14)--(t15)--(t11)--(t12)--(t13)--(t11);
    \draw (t11)--(t15);
     \end{tikzpicture}
\caption{The graphs $US_6$ and $UP_6$.}
\label{Fig:US_UP}
\end{figure}

\begin{theo}\label{theo:uni}
For any $n$-vertex unicyclic graph $G$ with $n\geq 3$, we have
$$
n(US_n)\geq n(G)\geq n(UP_n).
$$
\end{theo}

Generalizing the notations $US_n$ and $UP_n$, we define $US_n^l$ to be the graph obtained 
by attaching $n-l$ pendent vertices to one vertex of $C_{l}$ and $UP_n^l$ to be obtained by attaching 
a pendent path of length $n-l$ to a vertex of $C_l$. It is easy to see that $US_n=US_n^3$ and 
$UP_n=UP_n^3.$ Now for unicyclic graphs with a given girth we have the following.

\begin{theo}\label{theo:unigir}
If $G$ is an $n$-vertex unicyclic graph with girth $l$, then 
$$
n(US_n^l)\geq n(G)\geq n(UP_n^l)
$$
for any integer $n\geq l\geq 3.$
\end{theo}

Theorems~\ref{theo:uni} and \ref{theo:unigir}, compared with known results on the Wiener index (see for instance \cite{yu2010}), further confirming the negative correlation between $n(G)$ and $\W(G)$.

A path $P$ in $G$ is called a {\it segment} if the end vertices of $P$ are of degrees 1 or at least 3 in $G$, and the internal vertices of $P$ are all of degree 2 in $G$. The non-increasing sequence 
of the segment lengths in $G$ is called its {\it segment sequence}. For example, the segment 
sequences of $S_n$ and $P_n$ are $(1,1,\dots,1)$ and $(n-1)$, 
respectively. Extremal problems on trees with a given segment sequence have recently been considered \cite{lin2015segment,AWW2015,PrevPaper}. In particular, findings concerning $n(G)$ and $\W(G)$ further confirm the negative correlation between them.

Sometimes in specific arguments we may, when there is no confusion, allow a segment sequence to be not necessarily non-increasing. After all, the ordering of the entries does not change the set of segment lengths in a graph.

We denote by $\mathbb{U}(l_1,\dots,l_m)$ the set of all unicyclic graphs with segment sequence 
$(l_1,\dots,l_m)$. If $l_i\geq 3$, $U_i(l_1,\dots,l_m)$ is the element of $\mathbb{U}(l_1,\dots,l_m)$ 
whose cycle consists of one segment of length $l_i$ and the segments not on 
the cycle form a starlike tree. 
We then have the following for unicyclic graphs with a given segment sequence.

\begin{theo}
\label{Th:MainSTree}
Let $(l_1,\dots,l_m)$ be a segment sequence of an $n$-vertex unicyclic graph with $l_1\geq 3$.
For any $U\in \mathbb{U}(l_1,\dots,l_m)$, we have 
$$
n(U)\leq n(U_1(l_1,\dots,l_m)).
$$
\end{theo}

\begin{rem}
\label{Rem:H1H2}
It is interesting to note that the analogue of Theorem~\ref{Th:MainSTree}, with respect to the Wiener index, is not true. For instance, $U_1(4,4,1,1)$ does not minimize the Wiener index among all graphs in $\mathbb{U}(4,4,1,1)$ since
$$
\W(H_1)=118<120=\W(U_1(4,4,1,1))
$$
where $H_1$ is as shown in Figure~\ref{Fig:CounterE_W}.

\begin{figure}[ht!]
\centering
\begin{tikzpicture}[scale=.8]
\node[fill=black,circle,inner sep=1pt] (t0) at (0,0) {}; 
\node[fill=black,circle,inner sep=1pt] (t1) at (1,0) {};
\node[fill=black,circle,inner sep=1pt] (t2) at (2,1) {};
\node[fill=black,circle,inner sep=1pt] (t3) at (3,1) {};
\node[fill=black,circle,inner sep=1pt] (t4) at (4,1) {};
\node[fill=black,circle,inner sep=1pt] (t5) at (5,0) {};
\node[fill=black,circle,inner sep=1pt] (t6) at (6,0) {};
\node[fill=black,circle,inner sep=1pt] (t7) at (4,-1) {};
\node[fill=black,circle,inner sep=1pt] (t8) at (3,-1) {};
\node[fill=black,circle,inner sep=1pt] (t9) at (2,-1) {};
\draw (t0)--(t1)--(t2)--(t3)--(t4)--(t5)--(t7)--(t8)--(t9)--(t1);
\draw (t5)--(t6);
\node[fill=black,circle,inner sep=1pt] (h1) at (9-1.5,-1) {};
\node[fill=black,circle,inner sep=1pt] (h2) at (8-1.5,0) {};
\node[fill=black,circle,inner sep=1pt] (h3) at (9-1.5,1) {};
\node[fill=black,circle,inner sep=1pt] (h4) at (10-1.5,0) {};
\node[fill=black,circle,inner sep=1pt] (h5) at (10-1.5,1) {};
\node[fill=black,circle,inner sep=1pt] (h6) at (10-1.5,-1) {};
\node[fill=black,circle,inner sep=1pt] (h7) at (11-1.5,0) {};
\node[fill=black,circle,inner sep=1pt] (h8) at (12-1.5,0) {};
\node[fill=black,circle,inner sep=1pt] (h9) at (13-1.5,0) {};
\node[fill=black,circle,inner sep=1pt] (h10) at (14-1.5,0) {};
\draw (h4)--(h3)--(h2)--(h1)--(h4)--(h7)--(h8)--(h9)--(h10);
\draw (h5)--(h4)--(h6);
\node at (3,-1.5) {$H_1$};
\node at (9.5,-1.5) {$H_2=U_1(4,4,1,1)$};
     \end{tikzpicture} 
\caption{The graphs $H_1$ and $H_2=U_1(4,4,1,1)$ in Remark \ref{Rem:H1H2}.}
\label{Fig:CounterE_W}
\end{figure}
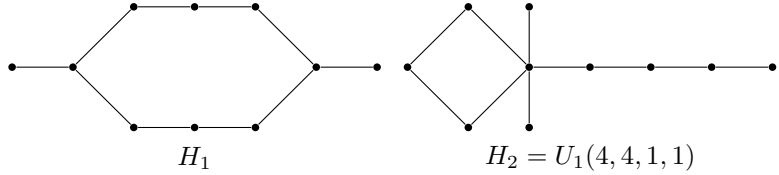

In other words, we have one of very few cases of extremal structures where the negative correlation between $n(G)$ and $\W(G)$ fails to hold. 
\end{rem}

When every segment is too short to form a cycle on its own, that is, when $l_1\leq 2$ where $U_1(l_1,\dots,l_m)$ is not a simple graph, the extremal unicyclic graph is  
$$U_{l_i,l_j}(l_{\sigma(1)},\dots,l_{\sigma(t)};l_{\sigma(t+1)},\dots,l_{\sigma(m-2)}),$$
where $\sigma:\{1,2,\dots,m-2\}\longrightarrow \{1,2,\dots,m\}\setminus\{i,j\}$ is the natural bijection that preserves the numerical ordering. This graph is obtained from the cycle $C_{l_i+l_j}$ by attaching $t$ pendent paths of lengths $l_{\sigma(1)},\dots,l_{\sigma(t)}$ at one vertex $v$ and $m-t-2$ pendent paths of lengths $l_{\sigma(t+1)},\dots,l_{\sigma(m-2)}$ to another vertex so that the graph has the segment sequence $(l_1,\dots,l_m)$. For example the graph $H_1$ in Figure \ref{Fig:CounterE_W} is $U_{4,4}(1;1).$ 
Noting that $U_{l_i,l_j}(l_{\sigma(1)},\dots,l_{\sigma(t)};l_{\sigma(t+1)},\dots,l_{\sigma(m-2)})$ fails to be a simple graph if $l_i=l_j=1$, the extremal unicyclic structure may also be the graph $U^1_n$, obtained by attaching a pendent vertex to each of two vertices of $C_3$ and $n-5$ pendent vertices to the third vertex.
\begin{theo}
\label{Th:Short_n}
Let $(l_1,\dots,l_m)$ be a segment sequence with $l_1 \leq 2$ and $H \in \mathbb{U}(l_1,\dots,l_m)$:
\begin{itemize}
\item[i)] If $l_1=2$ and $m\geq 4$, then 
$$
n(H) \leq n(U_{l_1,l_2}(l_3,\dots,l_{m-1};l_m)).
$$
\item[ii)] If $l_1=1$ and $m\geq 6$, then 
$$
n(H) \leq n(U^1_m).
$$
\end{itemize}
\end{theo}

Now let $\mathbb{U}_{n,m}$ denote the set of unicyclic graphs with $n$ vertices and $m$ 
segments. Again we have several possible extremal structures that maximize the number of subtrees, depending on the value of $n$ and $m$.

\begin{theo}\label{theo:segnum}
Let $G\in \mathbb{U}_{n,m}$ for some integers $n\geq 3$ with $n\geq m$:
\begin{itemize}
\item[i)] If $n\geq m+2$, then 
$$
n(G)\leq n(U_1(l_1,\dots,l_m)),
$$
where $(l_1,\dots,l_m)$ is a segment sequence of $n$-vertex unicyclic graphs with $2l_i+1\leq l_1\leq 2l_i+3$ for any $i\neq 1.$ 
%with equality if and only if $G=U_1(l_1,\dots,l_m)$. 
\item[ii)] If $n=m+1$, then $n(G)\leq n(U_{2,1}(\hspace{-0.1cm}\underbrace{1\dots,1}_{m-3\text{ times}\tiny}\hspace{-0.1cm};1))$.
\item[iii)] If $n=m$, then $n(G)\leq n(U^1_n)$.
\end{itemize}
\end{theo}

Similarly, for the number of independent vertex subsets, we have the following among unicyclic graphs with a given segment sequence.

\begin{theo}
\label{Th:MainInd}
Consider a segment sequence $L=(l_1,\dots,l_m)$ with $l_1 \geq 3$:
\begin{itemize}
\item If $L$ contains an even entry $l_i\geq 3$
and if $l_{i_0}$ is the smallest such entry, then
\begin{equation}
\label{Eq:SE}
\sigma(U_{i_0}(l_1,\dots,l_m))\geq \sigma(G)
\end{equation}
for any $G\in \mathbb{U}(l_1,\dots,l_m)$.
\item If all the $l_i$'s that are at least 3 are odd, then 
\begin{equation}
\label{Eq:SO}
\sigma(U_1(l_1,\dots,l_m))\geq \sigma(G)
\end{equation}
for any $G\in \mathbb{U}(l_1,\dots,l_m)$.
\end{itemize}
Equality in each of \eqref{Eq:SE} and \eqref{Eq:SO} holds if the two compared graphs are isomorphic.
\end{theo}

\begin{rem}
Again, the analogue of Theorem~\ref{Th:MainInd} for $\Z$ does not hold. For example,
$$
\Z(U_1(6,4))=114<115=\Z(U_2(6,4)).
$$
That is, among unicyclic graphs with a given segment sequence the extremal trees that maximize $\sigma(\cdot)$ are not necessarily the same as those that minimize $\Z(\cdot)$. Thus the well known negative correlation between $\sigma(\cdot)$ and $\Z(\cdot)$ also fails to hold in $\mathbb{U}(l_1,\dots,l_m)$.
\end{rem}

Similar conclusion as Theorem~\ref{Th:Short_n} holds for the number of independent vertex subsets when all segments are short.

\begin{theo}
\label{Th:Short_n_Sigma}
Let $(l_1,\dots,l_m)$ be a segment sequence with $l_1 \leq 2$ and $H \in \mathbb{U}(l_1,\dots,l_m)$:
\begin{itemize}
\item[i)] If $l_1=2$ and $m\geq 4$, then 
$$
\sigma(H)\leq\sigma(U_{l_1,l_2}(l_4,\dots,l_{m-1};l_3)).
$$
%where $U=U_{l_1,l_2}(l_4,\dots,l_{m-1};l_3)$.
\item[ii)] If $l_1=1$ and $m\geq 6$, then 
$$
\sigma(H)\leq \sigma(U^1_m).
$$
\end{itemize}
\end{theo}

Our last theorem provides a characterization of the unicyclic graph with fixed number of segments and maximum $\sigma(\cdot)$.

\begin{theo}\label{theo:segnumSig}
Let $G\in \mathbb{U}_{n,m}$ for some positive integers $n\geq m$.
\begin{itemize}
\item[i)] If $n\geq m+3$, then
$$
\sigma(U_2(\ell_1,\dots,\ell_m))\geq \sigma(G),
$$
where $\ell_1=n-m-2$, $\ell_2=4$ and $\ell_3=\dots=\ell_m=1$.  
%with equality if and only if $G=U_2(\ell_1,\dots,\ell_m)$. 
\item[ii)] If $n=m+2$, then $\sigma(U_{2,2}(\hspace{-0.1cm}\underbrace{1,\dots,1}_{m-3 \text{ times} \tiny}\hspace{-0.1cm};1))\geq\sigma(G)$. %where $l_1=\dots=l_{m-3}=1.$
\item[iii)] If $n=m+1$, then $\sigma(U_{2,1}(\hspace{-0.1cm}\underbrace{1,\dots,1}_{m-3 \text{ times} \tiny}\hspace{-0.1cm};1))\geq\sigma(G)$. %where $l_1=\dots=l_{m-3}=1.$
\item[iv)] If $n=m$, then $\sigma(U^1_n)\geq \sigma(G)$.
\end{itemize}
\end{theo}

In Section~\ref{sec:uni} we will present the proofs of Theorems~\ref{theo:uni} and \ref{theo:unigir}. We then consider the extremal structures, among unicyclic graphs with a given segment sequence or the number of segments with respect to the number of subtrees, showing Theorems~\ref{Th:MainSTree}, \ref{Th:Short_n}, \ref{theo:segnum} in Section~\ref{Sec:MSubT}. Section~\ref{Sec:MInd} is devoted to the number of independent subsets, where Theorems~\ref{Th:MainInd}, \ref{Th:Short_n_Sigma}, \ref{theo:segnumSig} are proven. We comment on our findings in Section~\ref{sec:Con}.

\section{Subtrees of unicyclic graphs}
\label{sec:uni}

First we introduce some notations and useful observations.  Let $\eta_k(G)$ be the set of $k$-vertex subtrees of $G$ and $n_k(G)=|\eta_k(G)|$. It is easy to see that $n_0(G)=1$, $n_1(G)=|V(G)|$ and $n(G)=\sum_{k\geq 0}n_k(G)$. We also define 
\begin{align*}
n_k(\mathcal{A},G)=\left|\{T\in \eta_k(G): \mathcal{A}\text{ is in }T\}\right|
\end{align*}
and 
$$
n(\mathcal{A},G)=\sum_{k\geq 0}n_k(\mathcal{A},G)
$$
where $\mathcal{A}$ can be any collection of vertices and/or edges of $G$. 

The following lemma lists some basic enumeration results.

\begin{lemma}
\label{Lem:Formulas}
For any integer $n\geq 2$, we have 
\begin{itemize}
\item $n_k(P_n)=n-k+1$ for any integer $n\geq k\geq 1$;
\item $n_k(C_n)=n$ for any integer $n\geq k\geq 1$;
\item $n_k(S_n)=\binom{n-1}{k-1}$ for any integer $n\geq k\geq 2$.
\end{itemize}
Hence 
$$
n(P_n)=\frac{n^2+n+2}{2},
$$
\begin{equation}
\label{Eq:Cn}
n(C_n)=n^2+1,
\end{equation}
and 
$$
n(S_n)=n+2^{n-1}.
$$
\end{lemma}
\begin{proof}
The proof follows from noticing that subtrees (with at least one vertex) in $P_n$ or $C_n$ are simply subpaths (including the ones on single vertices)  decided by its end vertices, and a subtree in $S_n$ containing the center is determined by the leaves it contains. We skip the details.
\end{proof}

The following extremal result on $n_k(G)$ among trees is useful to us.

\begin{theo}[\cite{book:1200029}]
\label{Th:Tree}
For any $n$-vertex tree $T$, we have
$$
n_k(S_n)\geq n_k(T)\geq n_k(P_n)
$$
for any integer $k\geq 0.$
\end{theo}

Of course, Theorem~\ref{Th:Tree} implies the extremality of $n(S_n)$ and $n(P_n)$. When considering subtrees containing a particular vertex, we have the following analogue of Theorem~\ref{Th:Tree}.

\begin{lemma}\label{lem:nk}
Let $c$ be the center of $S_n$, $\nu$ an end of $P_n$, and $v$ any vertex of an  $n$-vertex tree $T$, then
$$
n_k(c,S_n)\geq n_k(v,T)\geq n_k(\nu,P_n)
$$
for any integer $k\geq 0.$
\end{lemma}
\begin{proof}
The $k=0,1$ cases are trivial. For any $k\geq 2$ and any $v \in V(T)$, we have 
\begin{align*}
n_k(c,S_n)=n_k(S_n)\geq n_k(T)\geq n_k(v,T) \geq 1 = n_k(\nu,P_n).
\end{align*}
\end{proof}

\subsection{Unicyclic graphs}

First note that there is only one unicyclic graph on 3 vertices. For $n=4$, there are 
exactly two unicyclic graphs $US_4=UP_4$ and $C_4$, with $n(US_4)=n(C_4)$.

\begin{proof}[Proof of Theorem~\ref{theo:uni}]
We proceed by induction on $n$. Assume $n(US_n) \geq n(G) \geq n(UP_n)$ for any unicyclic graph $G$ on $n$ vertices, for $n\leq l$ for some $l\geq 3.$ For such a graph $G$ we first establish an analogue of Lemma~\ref{lem:nk} for unicyclic graphs. Let $c$ be the branching vertex (a vertex of degree at least 3) of $US_n$ 
and $\nu$ be the only vertex of degree $1$ in $UP_n$ (or one vertex of the cycle if $n=3$), then for any vertex $v$ in an $n$-vertex unicyclic graph $G$ we have 
$$
n(c,US_n)=n(US_n)-n-1\geq n(G)-n -1\geq n(v,G) \geq n+3 \geq n(\nu,UP_n).
$$
If $e$ is an edge on the cycle of $G$ chosen to be closest to $v$, then there are at least $n$ subtrees containing $v$ in $G-e$, and there 
are at least three subtrees of $G$ that contain both $v$ and $e$. So the second to last inequality is valid for any $n$ (not just for $n\leq l$).

Now, assume that $G$ is a $(n=l+1)$-vertex unicyclic graph. We consider two cases:
\begin{itemize}
\item First, suppose $G$ has a pendent vertex $v$ with neighbor $v'$, let $u$ be a leaf in $US_n$ and $c$ be its 
branching vertex. Let $\nu'$ be the neighbor of $\nu$ in $UP_n$. Then $G-v$ is still a unicyclic graph and 
\begin{align*}
n(US_n)
&=n(US_n-u)+n(u,US_n)\\
&=n(US_{n-1})+n(c,US_{n-1})+1\\
&\geq n(G-v)+n(v',G-v)+1=n(G) \\
&\geq n(UP_{n-1})+n(\nu',UP_{n-1}) + 1 \\
&=n(UP_n-\nu)+n(\nu,UP_n) = n(UP_n).
\end{align*}
\item If $G$ has no pendent vertex, then $G=C_n$. Direct computation shows, for $n\geq 4$
\begin{align*}
n(US_n) & = n_0(US_n) + n_1(US_n) + n_2(US_n) + \sum_{k=3}^n n_k(US_n) \\
&=1+n +n+\sum_{k=3}^n\left( \binom{n-1}{k-1}+2\binom{n-3}{k-3} \right) \\
&=2^{n-1}+2^{n-2}+n+1 \geq n^2+1 = n(C_n)
\end{align*}
and 
\begin{align*}
n(UP_n) 
&=\frac{n^2+n+2}{2}+3(n-3)=\frac{n^2+7n-16}{2} \leq n^2 +1 =n(C_n).
\end{align*}
\end{itemize}
\end{proof}

\subsection{Unicyclic graphs with fixed girth}

Instead of proving Theorem~\ref{theo:unigir} we show a stronger statement below, it is easy to see that Theorem~\ref{theo:unigir} follows as an immediate consequence.

\begin{theo}
If $G$ is an $n$-vertex unicyclic graph with girth $l$, then 
$$
n_k(US_n^l)\geq n_k(G)\geq n_k(UP_n^l)
$$
for any integers $k\geq 0$ and $n\geq l\geq 3.$
\end{theo}
\begin{proof}
We proceed by induction on $n$. When $n=l$ or $l+1$ there is only one $n$-vertex unicyclic graph with 
girth $l$. 

Assume that the inequalities hold whenever $l\leq n\leq t$ for some $t\geq l+1$. Now, we consider 
the case where $n=t+1\geq l+2$. As part of the induction hypothesis, we know that 
$$
n_k(US_t^l)\geq n_k(G)\geq n_k(UP_t^l)
$$
for any integers $k\geq 0$ and $3\leq l\leq t$, and any unicyclic graph $G$ on $t$ vertices with girth $l$. 

For any $v\in V(G)$, $c$ being the branching vertex of 
$US_t^l$ and $\nu$ being the leaf in $UP_t^l$, we have, for $k\geq 2$,
$$
n_k(c,US_t^l) = n_k(US_t^l)- n_k(P_{l-1}) \geq n_k(v,G),
$$
$$
n_k(v,G)\geq n_k(\nu,UP_n^l)=
\begin{cases}
1 \hspace{1.5cm}\text{ if }k\leq n-l+1,\\
k-(n-l) \text{ if }n-l+1\leq k\leq n,\\
0 \hspace{1.6cm}\text{otherwise}
\end{cases}
$$
and 
$$
n_k(v,G)\geq 
\begin{cases}
1 \hspace{0.9cm}\text{ if }k\leq d+1,\\
k-d \hspace{0.3cm}\text{ if }d+1\leq k\leq d+l,\\
l \hspace{1cm}\text{ if }d+l\leq k\leq n,\\
0 \hspace{1.1cm}\text{otherwise}
\end{cases}
$$
where $d \leq n-l$ is the distance between $v$ and the cycle in $G$. Consequently, for any $k\geq 2$, we have
$$
n_k(c,US_t^l)\geq n_k(v,G)\geq n_k(\nu,UP_t^l) . 
$$

With $n\geq l+2$, $G$ must have a leaf $v$ with a neighbor $v'$. Let 
$\nu$ be the leaf in $UP_n^l$ with neighbor $\nu'$, and $\mu$ be a leaf in $US_n^l$.
Then, for any $k\geq 2$ we have
\begin{align*}
n_k(US_n^l)
&=n_k(US_n^l-\mu)+n_k(\mu,US_n^l)
=n_k(US_{n-1}^l)+n_{k-1}(c,US_{n-1}^l)\\
&\geq n_k(G-v)+n_{k-1}(v',G-v)=n_k(G)\\
&\geq n_k(UP_{n-1}^l)+n_{k-1}(\nu',UP_{n-1}^l) \\
&=n_k(UP_n^l-\nu)+n_k(\nu,UP_n^l)=n_k(UP_n^l).
\end{align*}
\end{proof}

\section{Subtrees of unicyclic graphs with given segment sequences}
\label{Sec:MSubT}

In this section we consider the extremal unicyclic graphs, with a given segment sequence or number of segments, that maximize the number of subtrees.
We start with a technical Lemma.

\begin{lemma}
\label{Lem:FG_StarLike}
Let $C=C_n(l_1,l_2,\dots,l_m)$ be a unicyclic graph obtained by attaching $m$ 
pendent paths of lengths $l_1,l_2,\dots,l_m$ to vertices of $C_n$. Define 
$SC=SC_n(l_1,l_2,\dots,l_m)$ to be the special case where all $m$ paths are attached to a single vertex of $C_n$ (Figure~\ref{Fig:TransC_SC}). Then for any non-negative integer $k$ we have 
\begin{equation}\label{ineq:1}
n_k(C)\leq n_k(SC).
\end{equation}
Furthermore, if $c$ is the branching vertex of $SC$, then for any 
vertex $v$ on the cycle of $C$ we have 
\begin{equation}\label{ineq:2}
n_k(v,C)\leq n_k(c,SC).
\end{equation}
\end{lemma}
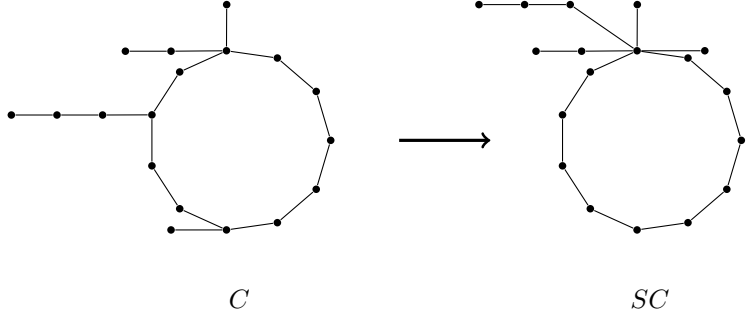
\begin{figure}[!ht]
\centering
    \begin{tikzpicture}[scale=0.6]
    \node[fill=black,circle,inner sep=1pt] (t1) at (1*2,0) {};
    \node[fill=black,circle,inner sep=1pt] (t2) at (0.841253532831*2, 0.540640817456*2) {};
    \node[fill=black,circle,inner sep=1pt] (t3) at (0.415415013002*2, 0.909631995355*2) {};
    \node[fill=black,circle,inner sep=1pt] (t4) at (-0.142314838273*2, 0.989821441881*2) {};
    \node[fill=black,circle,inner sep=1pt] (t5) at (-0.654860733945*2, 0.755749574354*2) {};
    \node[fill=black,circle,inner sep=1pt] (t6) at (-0.959492973614*2, 0.281732556841*2) {};
    \node[fill=black,circle,inner sep=1pt] (t7) at (-0.959492973614*2, -0.281732556841*2) {};
    \node[fill=black,circle,inner sep=1pt] (t8) at (-0.654860733945*2, -0.755749574354*2) {};
    \node[fill=black,circle,inner sep=1pt] (t9) at (-0.142314838273*2, -0.989821441881*2) {};
    \node[fill=black,circle,inner sep=1pt] (t10) at (0.415415013002*2, -0.909631995355*2) {};
    \node[fill=black,circle,inner sep=1pt] (t11) at (0.841253532831*2, -0.540640817456*2) {};
    \node[fill=black,circle,inner sep=1pt] (t12) at (-1.5, 1.97) {};
    \node[fill=black,circle,inner sep=1pt] (t13) at (-2.5, 1.97) {};
    \node[fill=black,circle,inner sep=1pt] (t14) at (-0.28, 3) {};
    \node[fill=black,circle,inner sep=1pt] (t15) at (-3, 0.56) {};
    \node[fill=black,circle,inner sep=1pt] (t16) at (-4, 0.56) {};
    \node[fill=black,circle,inner sep=1pt] (t17) at (-5, 0.56) {};
    \node[fill=black,circle,inner sep=1pt] (t18) at (-1.5, -1.98) {};
    \draw (t4)--(t12)--(t13);
    \draw (t4)--(t14);
    \draw (t6)--(t15)--(t16)--(t17);
    \draw (t9)--(t18);
    \draw(t1)--(t2)--(t3)--(t4)--(t5)--(t6)--(t7)--(t8)--(t9)--(t10)--(t11)--(t1);
    \draw[very thick,->] (3.5,0)--(5.5,0);
    \node[fill=black,circle,inner sep=1pt] (t11) at (9+1*2,0) {};
    \node[fill=black,circle,inner sep=1pt] (t12) at (9+0.841253532831*2, 0.540640817456*2) {};
    \node[fill=black,circle,inner sep=1pt] (t13) at (9+0.415415013002*2, 0.909631995355*2) {};
    \node[fill=black,circle,inner sep=1pt] (t14) at (9+-0.142314838273*2, 0.989821441881*2) {};
    \node[fill=black,circle,inner sep=1pt] (t15) at (9+-0.654860733945*2, 0.755749574354*2) {};
    \node[fill=black,circle,inner sep=1pt] (t16) at (9+-0.959492973614*2, 0.281732556841*2) {};
    \node[fill=black,circle,inner sep=1pt] (t17) at (9+-0.959492973614*2, -0.281732556841*2) {};
    \node[fill=black,circle,inner sep=1pt] (t18) at (9+-0.654860733945*2, -0.755749574354*2) {};
    \node[fill=black,circle,inner sep=1pt] (t19) at (9+-0.142314838273*2, -0.989821441881*2) {};
    \node[fill=black,circle,inner sep=1pt] (t110) at (9+0.415415013002*2, -0.909631995355*2) {};
    \node[fill=black,circle,inner sep=1pt] (t111) at (9+0.841253532831*2, -0.540640817456*2) {};
    \node[fill=black,circle,inner sep=1pt] (t112) at (9+-1.5, 1.97) {};
    \node[fill=black,circle,inner sep=1pt] (t113) at (9+-2.5, 1.97) {};
    \node[fill=black,circle,inner sep=1pt] (t114) at (9+-0.28, 3) {};
    \node[fill=black,circle,inner sep=1pt] (t115) at (9+1.25-3, 3) {};
    \node[fill=black,circle,inner sep=1pt] (t116) at (9+1.25-4, 3) {};
    \node[fill=black,circle,inner sep=1pt] (t117) at (9+1.25-5, 3) {};
    \node[fill=black,circle,inner sep=1pt] (t118) at (9+1.2, 1.98) {};
    \draw (t14)--(t112)--(t113);
    \draw (t14)--(t114);
    \draw (t14)--(t115)--(t116)--(t117);
    \draw (t14)--(t118);
    \draw(t11)--(t12)--(t13)--(t14)--(t15)--(t16)--(t17)--(t18)--(t19)--(t110)--(t111)--(t11);
    \node at (0, -3.5) {$C$};
    \node at (9, -3.5) {$SC$};
     \end{tikzpicture}
\caption{Examples of the graphs $C$ and $SC$ in Lemma \ref{Lem:FG_StarLike}.}
\label{Fig:TransC_SC}
\end{figure}
\begin{proof}
The case of $k=0$ or $1$ is obvious. For $k\geq 2$, we proceed by induction on $m$. If $m=1$, the inequality \eqref{ineq:1} is obvious as we have $C=SC$. Inequality \eqref{ineq:2} follows from the fact that there are more subtrees containing $c$ than those containing any other vertex $v$ on the cycle of $C=SC$.

Suppose that \eqref{ineq:1} and \eqref{ineq:2} hold for $m\leq s$ for some $s\geq 1$, and consider now 
the case $m=s+1$. Let $P$ be a pendent path of length $l_m$ with vertices $u$ on the cycle of $C$ and $v_1,v_2,\dots,v_{l_m}$ not on the cycle. Note that 
by the induction hypothesis 
$$
n_k(u,C-v_1-v_2-\dots-v_{l_m})\leq n_k(c,SC-v_1-v_2-\dots-v_{l_m})
$$
for any $k\geq 0.$ For any integer $k\geq 2$, we have 
\begin{align*}
n_k(C)&=n_k(C-v_1-v_2-\dots-v_{l_m})+n_k(P)\\
& \quad\quad\quad +\sum_{\substack{k_1+k_2=k+1\\k_1,k_2\geq 2}}n_{k_1}(u,C-v_1-v_2-\dots-v_{l_m})\cdot n_{k_2}(u,P)\\
&\leq n_k(SC-v_1-v_2-\dots-v_{l_m})+n_k(P)\\
& \quad\quad\quad +\sum_{\substack{k_1+k_2=k+1\\k_1,k_2\geq 2}}n_{k_1}(c,SC-v_1-v_2-\dots-v_{l_m})\cdot n_{k_2}(c,P) \\
& =n_k(SC)
\end{align*}
and 
\begin{align*}
n_k(v,C)&=n_k(v,C-v_1-v_2-\dots-v_{l_m})\\
& \quad\quad\quad +\sum_{\substack{k_1+k_2=k+1\\k_2\geq 2}}n_{k_1}(\{v,u\},C-v_1-v_2-\dots-v_{l_m})\cdot n_{k_2}(u,P)\\
&\leq n_k(v,C-v_1-v_2-\dots-v_{l_m})\\
& \quad\quad\quad +\sum_{\substack{k_1+k_2=k+1\\k_2\geq 2}}n_{k_1}(u,C-v_1-v_2-\dots-v_{l_m})\cdot n_{k_2}(u,P)\\
&\leq n_k(c,SC-v_1-v_2-\dots-v_{l_m})\\
& \quad\quad\quad +\sum_{\substack{k_1+k_2=k+1\\k_2\geq 2}}n_{k_1}(c,SC-v_1-v_2-\dots-v_{l_m})\cdot n_{k_2}(c,P) \\
& =n_k(c,SC)
\end{align*}
for any vertex $v$ on the cycle of $C$.
\end{proof}

Now we establish, through a series of observations, the characteristics of the extremal structure among unicyclic graphs with a given segment sequence. We first introduce an operation that, while preserving the segment sequence of a graph, generate more subtrees of any size. This is 
a slightly stronger version of Lemma 1 in \cite{PrevPaper}, the 
proof is similar.

\begin{lemma}
\label{Lem:TreeTrans}
Let $P$ be a path with end vertices $u$ and $v$. Let $G$ and $K$ be connected graphs (not necessarily trees), 
which contain vertices $u'$ and $v'$, respectively. 
Let $H$ be the graph obtained by merging $u$ with $u'$, and $v$ with $v'$.
Let $v_1,v_2,\dots,v_l$ be the neighbors of $v$ in $K$. Define 
$$H'=H-vv_1-vv_2-\dots-vv_l+uv_1+uv_2+\dots+uv_l$$
as in Figure \ref{Fig:Trans}. Then for any integer $k\geq 0$, we have
$$
n_k(H')\geq n_k(H)
$$
with equality if and only if $G$ or $K$ is a single vertex 
(in which case $H$ and $H'$ are isomorphic) or $k\in\{0,1,2\}$.
\end{lemma}

\begin{figure}[!ht]
\centering
    \begin{tikzpicture}[scale=0.75]
    \draw[dashed] (1,-1.25)--(0,1)--(2,1)--(1,-1.25);
    \draw[dashed] (4,-1.25)--(3,1)--(5,1)--(4,-1.25);
    \node[fill=black,label=above:{$u$},circle,inner sep=1pt] () at (1,0-1) {};
    \node[fill=black,circle,inner sep=1pt] () at (2,0-1) {};
    \node[fill=black,label=above:{$v$},circle,inner sep=1pt] () at (4,0-1) {};   
    \node[fill=black,circle,inner sep=0pt] () at (2.4,0-1) {};
    \node[fill=black,circle,inner sep=0pt] () at (2.7,0-1) {};
    \node[fill=black,circle,inner sep=0pt] () at (3,0-1) {};
    \node[fill=white,label=left:{$P$},circle,inner sep=0pt] () at (3,0.3-1) {};
    \node[fill=white,label=below:{$G$},circle,inner sep=0pt] () at (1,1.75-1) {};
    \node[fill=white,label=below:{$K$},circle,inner sep=0pt] () at (4,1.75-1) {};
    \draw[very thick,->] (5.5,0)--(6.5,0);
    \draw (1,0-1)--(2,0-1)--(2.25,0-1);
    \draw (4,0-1)--(3.25,0-1);
    \draw[dashed] (1+7,-1.25+1)--(0+7,1+1)--(2+7,1+1)--(1+7,-1.25+1);
    \draw[dashed] (8,0.2)--(7,-2)--(9,-2)--(8,0.2);
    \node[fill=black,label=above:{$u$},circle,inner sep=1pt] () at (7+1,0) {};
    \node[fill=black,circle,inner sep=1pt] () at (7+2,0) {};
    \node[fill=black,label=above:{$v$},circle,inner sep=1pt] () at (7+4,0) {};
    \node[fill=black,circle,inner sep=0pt] () at (7+2.4,0) {};
    \node[fill=black,circle,inner sep=0pt] () at (7+2.7,0) {};
    \node[fill=black,circle,inner sep=0pt] () at (7+3,0) {};
    \node[fill=white,label=left:{$P$},circle,inner sep=0pt] () at (7+3,0.3) {};
    \node[fill=white,label=left:{$H'$},circle,inner sep=0pt] () at (7+3,-2.5) {};
    \node[fill=white,label=left:{$H$},circle,inner sep=0pt] () at (3,-2.5) {};
    \node[fill=white,label=below:{$G$},circle,inner sep=0pt] () at (7+1,1.75) {};
    \node[fill=white,label=below:{$K$},circle,inner sep=0pt] () at (7+1,-1) {};
    \draw (7+1,0)--(7+2,0)--(7+2.25,0);
    \draw (7+4,0)--(7+3.25,0);
     \end{tikzpicture}
\caption{The graphs $H$ and $H'$ in Lemma \ref{Lem:TreeTrans}.}
\label{Fig:Trans}
\end{figure}
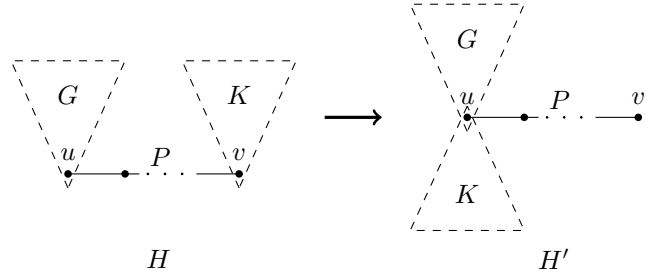
\begin{proof}
We construct a one-to-one map from the set of $k$-vertex subtrees of $H$ to that of $H'$:
\begin{itemize}
\item Any subtrees that contain no edge from $G$ or no edge from $K$ stay the same.
\item A subtree $S$ of $H$ that uses edges from both $G$ and $K$ is mapped to that of $H'$ spanned 
by the same set of vertices; such subtrees contain the entire $P$.
\end{itemize}

If $3\leq k\leq |H|$ and neither $G$ or $K$ is a single vertex, then $H'$ can have subtrees using edges from $G$ and $K$ but none 
from $P$. These subtrees do not have preimages in the above map.
\end{proof}

Lemma \ref{Lem:TreeTrans} can, of course, be stated for general $n(.)$ instead of $n_k(.)$. While doing so we also replace the path $P$ with a general graph $R$ in the statement.

\begin{lemma}
\label{Lem:GraphTrans}
Let $u$ and $v$ be vertices of a graph $R$, such that $n(u,R)\geq n(v,R)$. Let $G$ and $K$ be connected graphs, which contain vertices $u'$ and $v'$, respectively. 
Let $H$ be the graph obtained by merging $u$ with $u'$, and $v$ with $v'$, and $H'$ obtained by merging $u',v'$ and $u$. Then, we have
$$
n(H')\geq n(H)
$$
with equality if and only if $n(u,R)= n(v,R)$ and $K$ is a single vertex (in which case $H$ and $H'$ are isomorphic).
\end{lemma}
\begin{proof}
\begin{align*}
n(H')-n(H)=\big(n(u,R)-n(v,R)\big)n(v',K)+(n(u',G)-1)(n(v',K)-1).
\end{align*}
\end{proof}

The next observation introduces some simple conditions that can be used to compare the number of subtrees of different graphs.

\begin{lemma}\label{Lem:Cond}
Let $K$ be a graph with at least two vertices, one of which is $u$. Let $H$ and $H'$ be two other graphs with vertices $w$ and $w'$ respectively. Let $G$ be the graph obtained from $K$ and $H$ by identifying $u$ and $w$, and $G'$ the graph obtained from identifying $u$ and $w'$. Then
$$ n(G') \geq n(G) $$
if 
\begin{equation}\label{eq:cond}
n(H') \geq n(H) 
\hbox{ and }
n(w', H') \geq n(w, H).
\end{equation}
Furthermore, we have $n(G') > n(G)$ if strict inequality holds in at least one of the conditions in \eqref{eq:cond}.
\end{lemma}

\begin{proof}
Assuming \eqref{eq:cond}, we have
\begin{align*}
n(G') & = n(H') + n(K) - 2 + (n(w', H') -1)(n(u,K)-1) \\
& \geq n(H) + n(K) - 2 + (n(w, H) -1)(n(u,K)-1) \\
& = n(G)
\end{align*}
with equality if and only if both equalities hold in \eqref{eq:cond}.
\end{proof}

In the next few lemmas we compare the total number of subtrees by establishing the conditions \eqref{eq:cond} between certain pairs of subgraphs.

\begin{lemma}
\label{Lem:TwoSeg}
For any integers $l_1\geq l_2\geq 3$, we have
$$
n(U_1(l_1,l_2))\geq n(U_2(l_1,l_2))\qquad \text{and}
\qquad n(u,U_1(l_1,l_2))\geq n(v,U_2(l_1,l_2)),
$$
where $u$ and $v$ are the branching vertices of $U_1(l_1,l_2)$ and $U_2(l_1,l_2)$, respectively. The equality holds if and only if $l_1 = l_2$.
\end{lemma}

\begin{proof}
First, we choose an edge (on the cycle) incident to the branching vertex of 
$U_1(l_1,l_2)$. We count the number of subtrees by considering the ones that contain this edge and the others, yielding
\begin{align}
n(U_1(l_1,l_2))
&=n(P_n)+|\{(x,y,z): x,y,z\geq 0,  x+y\leq l_1-2\text{ and } z\leq l_2\}|\nonumber\\
&=n(P_n)+(l_2+1)\frac{l_1(l_1-1)}{2}\nonumber\\
\label{Eq:l1l2}
&=n(P_n)+\frac{l_1^2l_2-l_1l_2+l_1^2-l_1}{2}.
\end{align}
Similarly
\begin{align*}
n(U_2(l_1,l_2))
&=n(P_n)+\frac{l_1l_2^2-l_1l_2+l_2^2-l_2}{2}.
\end{align*}
Thus 
$$
n(U_1(l_1,l_2))-n(U_2(l_1,l_2))=(l_1-l_2)(l_1l_2+l_2+l_1-1)/2 \geq 0
$$
with equality if and only if $l_1 = l_2$. Through the same enumeration process, we also have
\begin{equation}\label{Eq:ul1l2}
n(u,U_1(l_1,l_2))
=(l_2+1)\frac{(l_1+1)l_1}{2},
\end{equation}
\begin{align*}
n(v,U_2(l_1,l_2))
&=(l_1+1)\frac{(l_2+1)l_2}{2}
\end{align*}
and thus 
$$
n(u,U_1(l_1,l_2)) - n(v,U_2(l_1,l_2))=(l_1+1)(l_2+1)(l_1-l_2))/2 \geq 0
$$
with equality if and only if $l_1=l_2$.
\end{proof}
The following lemma considers a specific case. We skip the proof that follows from direct calculations.
\begin{lemma}
\label{Lem:SpecialCase}
$$
n(U_1(4,3))< n(U_1(3,2,2))
\hbox{ and }
n(c,U_1(4,3))< n(b,U_1(3,2,2))
$$
where $c$ and $b$ are the branching vertices of $U_1(4,3)$ and $U_1(3,2,2)$ respectively.
\end{lemma}

The graphs compared in the following lemma do not have the same segment sequences. 
It claims that we can reduce the length of the cycle to form a new pendent segment, without reducing the total number of subtrees.
\begin{lemma}
\label{Lem:TwoSegCn}
Let $G$ be a connected graph with at least two vertices, one of which is $v$. 
Define $C=C_n(v,G)$ to be the graph obtained by merging $v$ with a vertex of 
the $n$-vertex cycle $C_n$, and $U=U_1(l_1,l_2,v,G)$ to be the graph obtained by merging $v$ with the branching vertex of $U_1(l_1,l_2)$.
If $n=l_1+ l_2$ with $l_1\geq l_2 >0$ and $l_1\geq 3,$ then
we have
$$
n(C)< n(U).%\hbox{ and } n(v,C) < n(v,U).
$$
\end{lemma}
\begin{proof}
Let $u$ and $z$ be the vertex of $C_n$ and $U_1(l_1,l_2)$, respectively, merged with $v$ to obtain $C$ and $U$. For the sake of generality we let $d$ be the distance between $z$ and the branching vertex of $U_1(l_1,l_2)$ and state our proof accordingly. Note that we have $d=0$ here and Remark \ref{Rem:Short_S} addresses cases with other possible values of $d$.

Note that $n(v,G)\geq 2$. Direct calculation and previously established facts  (see \eqref{Eq:Cn}, \eqref{Eq:l1l2} and \eqref{Eq:ul1l2}) yield
\begin{itemize}
\item $n(U_1(l_1,l_2))=\frac{n^2+n+2}{2}+\frac{l_1^2l_2-l_1l_2+l_1^2-l_1}{2}$;
\item $n(C_n)=n^2+1$;
\item $n(u,C_n) =n^2+1-n(P_{n-1}) =\frac{n^2+n}{2}$;
\item[] \hspace{-0.55cm}$
\begin{array}{rrl}
\bullet&\hspace{-0.2cm}n(z,U_1(l_1,l_2)) & =\frac{l_1^2+l_1}{2}+ l_2|\{(x,y):x,y\geq 0 \text{ and }x+y\leq l_1-d-1\}| \\
&& \quad\quad  +l_2|\{(x,y):x,y\geq 0 \text{ and }x+y\leq d-1\}| \\
&& =\frac{l_1^2+l_1}{2}+ l_2\frac{(l_1-d+1)(l_1-d)}{2}+l_2\frac{(d+1)d}{2}.
\end{array}
$
%\item $n(c,U_1(l_1,l_2))=(l_2+1)\frac{(l_1+1)l_1}{2}$.
\end{itemize}
Since $l_1> d+1$, we now have
\begin{align*}
n(z,U_1(l_1,l_2))-n(u,C_n)
&=l_2\frac{l_1(l_1-1-d)+d^2-l_2-1}{2}> 0 \text{ for }l_1\geq 3 .
%&=\frac{l_2(l_1(l_1-1)-l_2-1)}{2}> 0 \text{ for }l_1\geq 3.
\end{align*}
Consequently, as $l_1\geq 3$ and $ l_2\leq l_1$, we conclude that
\begin{align}
n(U)-n(C)
& = n(U_1(l_1,l_2))+(n(z,U_1(l_1,l_2))-1)(n(v,G)-1)\nonumber\\
& \quad\quad\quad\quad\quad\quad -n(C_n)-(n(u,C_n)-1)(n(v,G)-1)\nonumber\\
& \geq n(U_1(l_1,l_2))-n(C_n)+n(z,U_1(l_1,l_2))-n(u,C_n)\nonumber\\
& = \frac{n^2+n+2}{2}+\frac{l_1^2l_2-l_1l_2+l_1^2-l_1}{2}-n^2-1 \nonumber \\
& \quad\quad\quad\quad\quad\quad +l_2\frac{l_1(l_1-1-d)+d^2-l_2-1}{2}\nonumber\\
\label{Eq:Var_d}
& = l_2(l_1(l_1-d-2)+d^2-l_2)> 0\quad \text{for }d=0.
\end{align}
%and 
%\begin{align*}
%n(c,U)-n(u,C)
%& = n(c,U_1(l_1,l_2))+(n(c,U_1(l_1,l_2))-1)(n(v,G)-1)\\
%& \quad\quad\quad\quad\quad\quad -n(u,C_n)-(n(u,C_n)-1)(n(v,G)-1)> 0.
%\end{align*}
\end{proof}
\begin{rem}
\label{Rem:Short_S}
Although $d=0$ in Lemma \ref{Lem:TwoSegCn}, other values of $d$ will also occur when studying graphs with short segments. 
 
By verifying that  \eqref{Eq:Var_d} is non-negative for $d=0,$ for $d=1$ and $l_2=1$, and for $d=2$ and $l_2\leq 2$, we have $n(U)-n(C)\geq 0$ for $U$ and $C$ as defined in Lemma \ref{Lem:TwoSegCn}. These observations will be useful when considering graphs with short segments.
\end{rem}
%It is clear by iterative use of Lemma \ref{Lem:TwoSeg} that to ensure largest $n(.)$, we can use the longest segment to make a cycle:
We are now ready to prove Theorems~\ref{Th:MainSTree}, \ref{Th:Short_n} and \ref{theo:segnum}.

\begin{proof}[Proof of Theorem~\ref{Th:MainSTree}]
Let $U$ be the graph in $\mathbb{U}(l_1,\dots,l_m)$ with the maximum number of subtrees. By Lemma \ref{Lem:TreeTrans}, it suffices to consider the case where all pendent paths are attached to some vertices of a cycle.
Let $U'$ be obtained from $U$ by moving all the pendent path to one 
vertex, Lemma \ref{Lem:FG_StarLike} implies that 
$$
n(U)< n(U')
$$
unless $U=U'.$

If $U'\in \mathbb{U}(l_1,\dots,l_m)$ then we are done.

Note that, however, it is possible that $U'\notin \mathbb{U}(l_1,\dots,l_m)$ if $U$ had more than one segment 
in its cycle. Should this be the case we can apply Lemma \ref{Lem:TwoSegCn} and pull out of the cycle of $U'$ 
the segments of $U$ contained in it, one at a time, to form pendent paths. 
At each step, the number of subtrees increases.

Let $U''$ be the graph obtained when we cannot do this anymore. Then the cycle in $U''$ consists of:
\begin{itemize}
 \item[(i)] two segments of length $2$ of $U$;
 \item[(ii)] or three segments of length $1$ of $U$, or two segments of $U$ of lengths 1 and 2 respectively;
 \item[(iii)] or just one segment of $U$.
\end{itemize}

\begin{itemize}
\item In the case of (i):
\begin{itemize}
\item If $U$ had a segment of length at least $4$, then by 
Lemmas~\ref{Lem:Cond} and \ref{Lem:TwoSeg} we can replace the cycle of $U''$ with the longest segment and then replace 
the new length-4 pendent segment by two segments of length $2$; 
\item Otherwise the maximum length of any segment in $U$ is $3$, and by Lemmas~\ref{Lem:Cond} and 
\ref{Lem:SpecialCase} we can replace $U_1(4,3)$ by 
$U(3,2,2)$ (while leaving the rest of the graph the same).
\end{itemize}
Either way we can only increase the number of subtrees.
\item In the case of (ii), we may simply replace the cycle $C_3$ with one formed by the longest segment, and replace 
the new pendent segment of length $3$ by three segments of length $1$, or one segment of length $1$ and another of length $2$. It is easy to see that, in these cases, we can only increase the number of subtrees.
\item In the case of (iii), we only need to replace the cycle with the longest segment. Again, we can only increase the number of subtrees through doing so.
\end{itemize}

In any of these scenarios we end up with $U'''=U_1(l_1,l_2,\dots,l_m) \in \mathbb{U}(l_1,\dots,l_m)$ that maximizes the number of subtrees. 
\end{proof}

\begin{proof}[Proof of Theorem \ref{Th:Short_n}]
Suppose that $(l_1,\dots,l_m)$ is a segment sequence of an unicyclic graph. Let $G\in \mathbb{U}(l_1,\dots,l_m)$ be with girth $g$. 

\noindent {\bf Case 1:} Suppose that $l_1=l_2=2$. If $g$ is at least $l_1+l_2+l_m$, then by Lemmas \ref{Lem:FG_StarLike}, \ref{Lem:TreeTrans} and \ref{Lem:TwoSegCn} we have $n(G)\leq n(U_1(l_1+l_2+l_m,l_3,\dots,l_{m-1}))$. In this case $l_1+l_2+l_m\in\{5,6\}$.  From Remark \ref{Rem:Short_S}, we have  $n(U_1(l_1+l_2+l_m,l_3,\dots,l_{m-1}))\leq U_{l_1,l_2}(l_3,\dots,l_{m-1};l_m)$. Now consider the case where the girth $g$ is strictly less than $l_1+l_2+l_m$.
\begin{itemize}
\item If $g=5$, then we must have $l_m=1$. The result once again follows from Lemma \ref{Lem:TwoSegCn} and Remark \ref{Rem:Short_S}.

\item If $g=4$ and the cycle consists of two segments (of length $2$), then our claim follows from Lemma \ref{Lem:GraphTrans}. Note that if $u$ and $v$ are the two branching vertices in $U_{2,2}(2;1)$, where the pendent segment of length $2$ is attached to $u$, then $n(u, U_{2,2}(2,1))>n(v, U_{2,2}(2,1)).$

\item If $g=3$, then $l_m=1$, then the claim follows from applications of Lemma \ref{Lem:GraphTrans} and the facts that

$$n(U_{2,2}(l,1))=(l^2+33l+54)/2>(l^2+29l+50)/2=n(U_{2,1}(l,2))$$ 
and 
$$n(u,U_{2,2}(l,1))=16(l+1)>14(l+1)=n(v,U_{2,1}(l,2)), $$ 
where the pendent path of length $l$ in $U_{2,2}(l,1)$ (resp. $U_{2,1}(l,2)$) is attached at $u$ (resp. $v$).
\end{itemize}

\noindent {\bf Case 2:} If $l_1=2,l_2=1$, then $l_m=1$ and the argument is similar to Case 1.

\noindent {\bf Case 3:} Now we assume that $l_1=1$:
\begin{itemize}
\item Suppose that $g\geq 5$. Then $n\geq 10$ and $n(G)< n(U_1(5,1,\dots,1))$. Since $n(U_{2,1}(1;1))=28>26=n(C_5)$ and $n(v,U_{2,1}(1;1))=17>15=n(u,C_5)$ where $u$ is an arbitrary vertex of $C_5$ and $v$ the vertex of degree $2$ in $U_{2,1}(1;1)$, we have $n(U_1(5,1,\dots,1))<n(U^1_m)$ by Lemma~\ref{Lem:Cond}.
\item Suppose that $g=4$, then $n\geq 8$. We can assume (by Lemma \ref{Lem:GraphTrans}) that three of the four branching vertex of $G$ have degree 3. Then we have
$$
n(G)=12\cdot 2^{n-5}+2^{n-7}+n+19< n+6+17\cdot 2^{n-5}=n(U^1_m) \text{ for }n\geq 8.
$$
\item The case of $g=3$ follows from direct application of Lemma \ref{Lem:GraphTrans}.
\end{itemize}
\end{proof}

\begin{proof}[Proof of Theorem~\ref{theo:segnum}]
Cases (ii) and (iii) are immediate consequences of Theorem \ref{Th:Short_n}.

For case (i), first note that the condition $2l_i+1 \leq l_1 \leq 2l_i+3$ for any $i\neq 1$ also implies that $|l_i-l_j|\leq 1$  for any $i,j\in \{2,\dots,m\}$, and the condition $n\geq m+2$ allows us to have a segment of length at least $3$.

\textbf{Case 1:} If the length of the longest segment is at least $3$, then by Theorem~\ref{Th:MainSTree} we may assume our extremal unicyclic graph on $n$ vertices with $m$ segments is of the form $U_1(l_1,\dots,l_m)$. Let $b(U_1(l_1,\dots,l_m))$ denote the branching vertex in this graph and $S_t$  be the $t$-th segment with length $l_t$, for any $t$.

\begin{itemize}
\item For any $2 \leq i \leq m$, if $l_1\leq 2l_i$, direct calculation as before shows
$$
n(U_1(l_1+1,l_i-1))> n(U_1(l_1,l_i)) 
$$
and 
$$
n(b(U_1(l_1+1,l_i-1)),U_1(l_1+1,l_i-1))\geq n(b(U_1(l_1,l_i)),U_1(l_1,l_i)). 
$$
Hence by Lemma~\ref{Lem:Cond}, replacing $S_1$ and $S_i$ with segments of length $l_1+1$ and $l_i-1$ would increase  the number of subtrees.
\item If $l_1\geq 2l_i+4>2l_i+3$ for some $2 \leq i \leq m$, similarly we have
$$
n(U_1(l_1-1,l_i+1))\geq n(U_1(l_1,l_i)) 
$$
and 
$$
n(b(U_1(l_1-1,l_i+1)),U_1(l_1-1,l_i+1))> n(b(U_1(l_1,l_i)),U_1(l_1,l_i)). 
$$
Lemma~\ref{Lem:Cond} implies that replacing $S_1$ and $S_i$ with segments of length $l_1-1$ and $l_i+1$ would increase  the number of subtrees.
\end{itemize}

\textbf{Case 2:} Now we consider the case where the longest segment has length $2$. After repeated applications of Lemmas \ref{Lem:FG_StarLike}, \ref{Lem:TreeTrans} and \ref{Lem:TwoSegCn} we have $U_1(l_1,\dots,l_{m-1})$, where $l_1\in\{3,4\}$, that has at least as many subtrees as before.

 If $l_1=4$, then $n(U_1(3,1,l_2,\dots,l_{m-1}))>n(U_1(l_1,\dots,l_{m-1})).$
 Otherwise $l_1=3$ and there has to be $i_0\in\{2,\dots,m-1\}$ with $l_{i_0}=2$. In this case we have
$$
n(U_1(l_1,\dots,l_{i_0-1},1,1,l_{i_0+1},\dots,l_{m-1}))>n(U_1(l_1,\dots,l_m)).
$$

Both Case 1 and Case 2 led to extremal structures as described in case (i) of Theorem~\ref{theo:segnum}.

\end{proof}

\section{Independent subsets of unicyclic graphs}
\label{Sec:MInd}

We now consider the maximum $\sigma(\cdot)$ in unicyclic graphs with a given segment sequence or number of segments. We start with recalling some of the previous established results.

As in Figure \ref{Fig:Slide}, we define $P(n, k, G, v)$ to be the graph obtained by merging a vertex $v$ of $G$ with the $k$-th vertex on $P_n$. The following is known as the ``Sliding Lemma'' \cite{zhao2006fibonacci,wagner2007extremal}.
\begin{lemma}[\cite{zhao2006fibonacci,wagner2007extremal}]
\label{Lem:Sliding}
Let $n$ be a positive integer, and write it as 
$n = 4m + h,$ for some $h \in \{1, 2, 3, 4\}$ and for some $m\in\mathbb{N}.$ Then 
\begin{align*}
\sigma(P(n, 2, G, v))  > \sigma(P(n, 4, G, v)) > \cdots  > \sigma(P(n, 2m + 2l, G, v)) \\
>\sigma(P(n, 2m + 1, G, v)) > \cdots  >\sigma(P(n, 3, G, v))  >\sigma(P(n, 1, G, v))
\end{align*}
and 
\begin{align*}
\Z(P(n, 2, G, v))  <\Z(P(n, 4, G, v)) < \cdots  <\Z(P(n, 2m + 2l, G, v)) \\
<\Z(P(n, 2m + 1, G, v)) < \cdots  <\Z(P(n, 3, G, v))  <\Z(P(n, 1, G, v))
\end{align*}
where $l = \lfloor \frac{h-1}{2}\rfloor$.
\end{lemma}

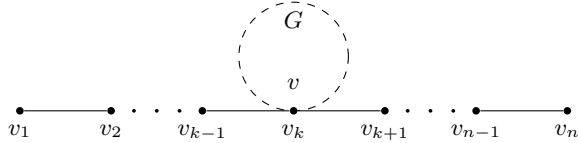
\begin{figure}[ht!]
 \centering
    \begin{tikzpicture}[scale=1.2]
        \node[fill=black,label=below:{\small$v_{k}$},circle,inner sep=1pt] (t0) at (0,0) {}; 

        \node[fill=black,label=below:{\small$v_{k+1}$},circle,inner sep=1pt] (t1) at (1,0) {};
        \node[fill=black,circle,inner sep=0.5pt] (t2) at (1.25,0) {};
        \node[fill=black,circle,inner sep=0.5pt] (t3) at (1.50,0) {};
        \node[fill=black,circle,inner sep=0.5pt] (t4) at (1.75,0) {};
        \node[fill=black,label=below:{\small$v_{n-1}$},circle,inner sep=1pt] (t5) at (2,0) {};
        \node[fill=black,label=below:{\small$v_n$},circle,inner sep=1pt] (t6) at (3,0) {};

        \node[fill=black,label=below:{\small$v_{k-1}$},circle,inner sep=1pt] (t7) at (-1,0) {};
        \node[fill=black,circle,inner sep=0.5pt] (t8) at (-1.25,0) {};
        \node[fill=black,circle,inner sep=0.5pt] (t9) at (-1.50,0) {};
        \node[fill=black,circle,inner sep=0.5pt] (t10) at (-1.75,0) {};
        \node[fill=black,label=below:{\small$v_2$},circle,inner sep=1pt] (t11) at (-2,0) {};
        \node[fill=black,label=below:{\small$v_1$},circle,inner sep=1pt] (t12) at (-3,0) {};

        \node[fill=white,label=below:{\small$G$},circle,inner sep=1pt] (t13) at (0,1.25) {};
        \node[fill=white,label=below:{\small$v$},circle,inner sep=1pt] (t14) at (0,0.5) {};
        
        \draw (t0)--(t1);
        \draw (t0)--(t7);
        \draw (t5)--(t6);
        \draw (t11)--(t12);
	\draw[dashed] (0,0.6) circle (0.6cm);
     \end{tikzpicture} 
\caption{The graph $P(n, k, G, v)$.}
\label{Fig:Slide}
\end{figure}

\begin{lemma}[\cite{Liu2007183}]
\label{Lem:Gather_Br}
Let $K, H_1$ and $H_2$ be connected graphs, such that $v,u\in V(K)$, $v'\in V(H_1)$, $u'\in V(H_2)$ and $E(H_i)\neq \emptyset$ for $i=1,2$. Let (as in Figure \ref{Fig:Gather_Br}) 
\begin{itemize}
\item $G$ be the graph obtained from merging $u$ and $v$ with $u'$ and $v'$ respectively;
\item $G_1$ be the graph obtained from merging $v,u'$ and $v'$; and
\item $G_2$ the graph obtained from merging $u,u'$ and $v'$.
\end{itemize}
Then $\sigma(G)<\max\{\sigma(G_1),\sigma(G_2)\}$.
\end{lemma}
\begin{figure}
\centering
\begin{tikzpicture}[scale=1]
\draw[dashed] (0,0) ellipse (0.5cm and 1cm);
\draw[dashed] (0.5,0)--(1.5,-0.1)--(1.5,-1)--(0.5,0);
\draw[dashed] (-0.5,0)--(-1.5,-0.1)--(-1.5,-1)--(-0.5,0);
\node at (-1.2,-0.4) {$H_1$};
\node at (1.2,-0.4) {$H_2$};
\node[ball color=black,circle,inner sep=1pt] () at (0.5,0) {};
\node[ball color=black,circle,inner sep=1pt] () at (-0.5,0) {};
%%%
\draw[dashed] (0+4,0+1.1) ellipse (0.5cm and 1cm);
\draw[dashed] (0.5+4,0+1.1)--(1.5+4,-0.1+1.1)--(1.5+4,-1+1.1)--(0.5+4,0+1.1);
\draw[dashed] (0.5+4,0+1.1)--(1.5+4,0.1+1.1)--(1.5+4,1+1.1)--(0.5+4,0+1.1);
\node at (1.2+4,0.4+1.1) {$H_1$};
\node at (1.2+4,-0.4+1.1) {$H_2$};
\node[ball color=black,circle,inner sep=1pt] () at (0.5+4,0+1.1) {};
\node[ball color=black,circle,inner sep=1pt] () at (-0.5+4,0+1.1) {};
%%%
\draw[dashed] (0+5.5,0-1.1) ellipse (0.5cm and 1cm);
\draw[dashed] (-0.5+5.5,0-1.1)--(-1.5+5.5,-0.1-1.1)--(-1.5+5.5,-1-1.1)--(-0.5+5.5,0-1.1);
\draw[dashed] (-0.5+5.5,0-1.1)--(-1.5+5.5,0.1-1.1)--(-1.5+5.5,1-1.1)--(-0.5+5.5,0-1.1);
\node at (-1.2+5.5,0.4-1.1) {$H_1$};
\node at (-1.2+5.5,-0.4-1.1) {$H_2$};
\node[ball color=black,circle,inner sep=1pt] () at (0.5+5.5,0-1.1) {};
\node[ball color=black,circle,inner sep=1pt] () at (-0.5+5.5,0-1.1) {};
\draw[very thick,->] (2,0.5)--(3,1); 
\draw[very thick,->] (2,-0.5)--(3,-1);
\node at (0,-1.5) {$G$};
\node at (0,0) {$K$};
\node at (6.5,1) {$G_1$};
\node at (4,1) {$K$};
\node at (6.5,-1) {$G_2$};
\node at (5.5,-1) {$K$};
\end{tikzpicture}
\caption{Graphs described in Lemma \ref{Lem:Gather_Br}.}
\label{Fig:Gather_Br}
\end{figure}
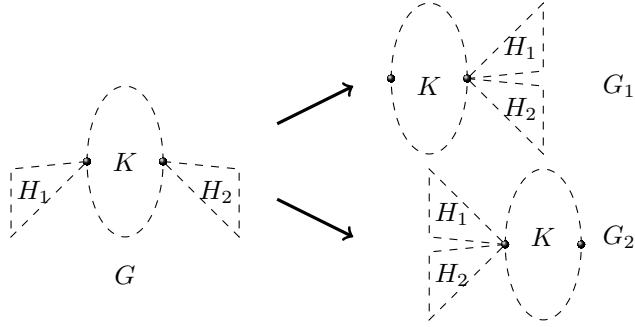

In particular, if $K$ is a path then the transformation in Lemma \ref{Lem:Gather_Br} corresponds to that of Lemma \ref{Lem:TreeTrans}. By choosing the $H_i$'s to be pendent segments in a unicyclic graph, repeated applications of Lemma \ref{Lem:Gather_Br} achieves the transformation in Lemma \ref{Lem:FG_StarLike}. 

\subsection{Unicyclic graphs with a given segment sequence}

Restricting our attention to unicyclic graphs with a given segment sequence, we introduce a few more technical observations. In the rest of paper we will also use the notation $[u]$ for the closed neighborhood of a vertex $u$.

\begin{lemma}
\label{Lem:Pull}
For any $l_1,l_2,\dots,l_n$ with $l_i>3$ for some $i\in\{1,2,\dots,n\}$, we have
$$
\sigma(U_i(l_1,l_2,\dots,l_n))<\sigma(U_i(l_1,l_2,\dots,l_{i-1},l_{i_1},l_{i_2},l_{i+1},\dots,l_n)),
$$
where $l_{i_1}\geq 3, l_{i_2}\geq 1, l_{i_1}+l_{i_2}=l_i$ and $U_i(l_1,l_2,\dots,l_{i-1},l_{i_1},l_{i_2},l_{i+1},\dots,l_n)$ is the unicyclic graph obtained by merging one vertex of the cycle $C_{l_{i_1}}$ with the center of the starlike graph with segment sequence $(l_1,l_2,\dots,l_{i-1},l_{i_2},l_{i+1},\dots,l_n)$.
\end{lemma}
\begin{proof}
Let $u$ (resp. $u'$) be a vertex on the cycle of $U=U_i(l_1,l_2,\dots,l_n)$ (resp. $U'=U_i(l_1,l_2,\dots,l_{i-1},l_{i_1},l_{i_2},l_{i+1},\dots,l_n)$) that is adjacent to the branching vertex of $U$ (resp. $U'$). 
Then Lemma \ref{Lem:Sliding} implies that $\sigma(U'-u')>\sigma(U-u)$. 
Since $U'-[u']$ is a spanning subgraph of $U-[u]$, we obtain
\begin{align*}
\sigma(U)=\sigma(U-u)+\sigma(U-[u])<\sigma(U'-u')+\sigma(U'-[u'])= \sigma(U').
\end{align*}
\end{proof}

Lemma~\ref{Lem:Pull} essentially allows one to ``pull'' segments from a long cycle and increase $\sigma(\cdot)$ at the same time.
The next lemma treats the cases when one can no longer pull any segments from the cycle. This happens if the remaining cycle consists of two segments of length $2$, or two segments of length $1$ and $2$, or three segments of length $1$.
We will explore the change in $\sigma(\cdot)$ when this cycle (of length 3 or 4, after another cycle takes its place) is split into the corresponding segments.
Following the notations of Lemma \ref{Lem:TwoSegCn}, we let $U_i(l_1,l_2,\dots,l_n,v,G)$ be the graph obtained by merging a vertex $v$ of $G$ with the branching vertex of $U_i(l_1,l_2,\dots,l_n)$, and
$S(l_1,l_2,\dots,l_n,v,G)$ be obtained by attaching the paths $P_{l_1+1},P_{l_2+1},\dots,P_{l_n+1}$ at the vertex $v$ of $G$.

\begin{lemma}
\label{Lem:Pull_SCase}
For any vertex $v$ in a graph $G$ and $l\geq 3$, we have
\begin{align*}
\sigma(U_2(l,4,v,G))<\sigma(U_1(l,2,2,v,G)),\quad 
\sigma(U_2(l,3,v,G))<\sigma(U_1(l,2,1,v,G))
\end{align*}
and 
$$
\sigma(U_2(l,3,v,G))<\sigma(U_1(l,1,1,1,v,G)).
$$
\end{lemma}
\begin{proof}
From the fact $\sigma(P_{n+2})=\sigma(P_{n+1})+\sigma(P_n)$ for all $n\geq 0$, it is easy to show by induction that $3\sigma(P_{n+1})>4\sigma(P_n)$ for all $n\geq 0$.

In the calculations below, $w$ is chosen to be a vertex adjacent to $v$, not in $G$, but on the cycle. 
Then by Lemma \ref{Lem:Sliding} we have
\begin{align*}
\sigma(U_2(l,3,v,G))
&=\sigma(U_2(l,3,v,G)-w)+\sigma(U_2(l,3,v,G)-[w])\\
&=\sigma(S(l,1,v,G))+\sigma(P_l\cup (G-v))\\
&<\sigma(S(l-2,2,1,v,G))+\sigma(P_1 \cup P_2\cup P_{l-3}\cup (G-v))\\
&=\sigma(U_1(l,2,1,v,G))
\end{align*}
and
\begin{align*}
\sigma(U_1(l,1,1,1,v,G))
&=\sigma(U_1(l,1,1,1,v,G)-w)+\sigma(U_1(l,1,1,1,v,G)-[w])\\
&=\sigma(S(l-2,1,1,1,v,G))+\sigma(P_1\cup P_1\cup P_1 \cup P_{l-3}\cup (G-v))\\
&>\sigma(S(l,1,v,G))+\sigma(P_{l}\cup (G-v))=\sigma(U_2(l,3,v,G)).
\end{align*}
%since $6\sigma(P_{l-3})-\sigma(P_l)=3\sigma(P_{l-3})-2f(l)>0$, where $f(l)=1$ if $l=3$ and $f(l)=\sigma(P_{l-4})$ if $l\geq 4$.
Similarly we have
\begin{align*}
\sigma(U_2(l,4,v,G))
&=\sigma(S(l,2,v,G))+\sigma(P_1\cup P_l\cup (G-v))
\end{align*}
and
\begin{align*}
\sigma(U_1(l,2,2,v,G))
&=\sigma(S(l-2,2,2,v,G))+\sigma(P_{l-3}\cup P_2\cup P_2 \cup G-v) .
\end{align*}
Consequently
\begin{align*}
\sigma(U_2(l,4,v,G))-\sigma(U_1(l,2,2,v,G))
&<(2\sigma(P_l)-9\sigma(P_{l-3}))\sigma(G-v)\\
&=(6\sigma(P_{l-3})+4\sigma(P_{l-4})-9\sigma(P_{l-3}))\sigma(G-v)\\
&=(4\sigma(P_{l-4})-3\sigma(P_{l-3}))\sigma(G-v)<0\text{ if }l\geq 4.
\end{align*}
When $l=3$,
\begin{align*}
\sigma(U_2(3,4,v,G))
&=\sigma(P_3\cup P_3\cup (G-v))+\sigma(P_2\cup P_1\cup (G-[v]))\\
&=25\sigma(G-v)+6\sigma(G-[v]) \\
&\leq 27\sigma(G-v)+4\sigma(G-[v])\\
&=\sigma(P_2\cup P_2\cup P_2\cup (G-v))+\sigma(P_1\cup P_1\cup (G-[v]))\\
&=\sigma(U_1(3,2,2,v,G)).
\end{align*}
\end{proof}

One more technical lemma is needed, then we are ready to prove Theorem~\ref{Th:MainInd}.
\begin{lemma}
\label{Lem:Paths}
If $n\geq m\geq 1$ are integers, then
$$
\sigma(P_n\cup P_{m-1})\leq\sigma(P_{n-1}\cup P_{m})
\hbox{ if $m$ is odd,}
$$ 
and 
$$
\sigma(P_n\cup P_{m-1})\geq\sigma(P_{n-1}\cup P_{m})
\hbox{ if $m$ is even.}
$$
\end{lemma}
\begin{proof}
The case $n=m$ is trivial. For the rest of the proof we assume $n>m$.
For $m=1$ and $n\geq 2$ we have
\begin{align*}
\sigma(P_n\cup P_{m-1})
&=\sigma(P_n) =\sigma(P_{n-1})+\sigma(P_{n-2}) \\
& \leq 2\sigma(P_{n-1})=\sigma(P_{n-1}\cup P_{m}).
\end{align*}
For $m=2$ we have
\begin{align*}
\sigma(P_n\cup P_{m-1})
&=\sigma(P_n\cup P_{1})
=2\sigma(P_n)=2\sigma(P_{n-1})+2\sigma(P_{n-2})\\
&\geq 3\sigma(P_{n-1})=\sigma(P_{n-1}\cup P_{m}).
\end{align*}

We now proceed by induction on $m$:
\begin{itemize}
\item If $m$ is odd, then $m-1$ is even, by the induction hypothesis we have
$$
\sigma(P_{n-2})\sigma(P_{m-1})\leq \sigma(P_{n-1})\sigma(P_{m-2}).
$$
Thus
\begin{align*}
\sigma(P_n\cup P_{m-1})
&=\sigma(P_n)\sigma(P_{m-1})
=\sigma(P_{n-1})\sigma(P_{m-1})+\sigma(P_{n-2})\sigma(P_{m-1})\\
&\leq \sigma(P_{n-1})\sigma(P_{m-1})+\sigma(P_{n-1})\sigma(P_{m-2})
=\sigma(P_{n-1}\cup P_{m}).
\end{align*}
\item If $m$ is even, then similarly we have
\begin{align*}
\sigma(P_n\cup P_{m-1})
&=\sigma(P_n)\sigma(P_{m-1})
=\sigma(P_{n-1})\sigma(P_{m-1})+\sigma(P_{n-2})\sigma(P_{m-1})\\
&\geq \sigma(P_{n-1})\sigma(P_{m-1})+\sigma(P_{n-1})\sigma(P_{m-2})
=\sigma(P_{n-1}\cup P_{m}).
\end{align*}
\end{itemize}
\end{proof}

%We are now ready to prove Theorem~\ref{Th:MainInd}.

\begin{proof}[Proof of Theorem~\ref{Th:MainInd}]
From Lemmas \ref{Lem:Gather_Br}, \ref{Lem:Pull} and \ref{Lem:Pull_SCase}
we know that to find the graph with maximum $\sigma$ in $\mathbb{U}(l_1,\dots,l_m)$, 
we only need to consider graphs of the form $U_i(l_1,\dots,l_m).$

Let $\mathbb{I}_m=\{1,\dots,m\}$. Note that 
\begin{align*}
\sigma(U_i(l_1,\dots,l_m))
&=\sigma(P_{l_i-1})\prod_{\substack{j\in \mathbb{I}_m\\j\neq i}}\sigma(P_{l_j})
+\sigma(P_{l_i-3})\prod_{\substack{j\in \mathbb{I}_m\\j\neq i}}\sigma(P_{l_j-1}).
\end{align*}

Assume that there exist even entries in $(l_1,\dots,l_m)$, and $l_{i_0}$ is the smallest of them.
If $i\in \mathbb{I}_m$ and $i<i_0$ (i.e. $l_i\geq l_{i_0}$), then by Lemma \ref{Lem:Paths} we have

\begin{align*}
&\sigma(U_{i_0}(l_1,\dots,l_m))-\sigma(U_{i}(l_1,\dots,l_m))\\
&\hspace{2cm}=(\sigma(P_{l_{i_0}-1})\sigma(P_{l_{i}})-\sigma(P_{l_{i_0}})\sigma(P_{l_{i}-1}))\prod_{\substack{j\in \mathbb{I}_m\\j\neq i,i_0}}\sigma(P_{l_j})\\
&\hspace{2cm} \quad\quad\quad +(\sigma(P_{l_{i_0}-3})\sigma(P_{l_{i}-1})-\sigma(P_{l_{i_0}-1})\sigma(P_{l_{i}-3}))\prod_{\substack{j\in \mathbb{I}_m\\j\neq i,i_0}}\sigma(P_{l_j-1})\\
&\hspace{2cm}=(\sigma(P_{l_{i_0}-1})\sigma(P_{l_{i}})-\sigma(P_{l_{i_0}})\sigma(P_{l_{i}-1}))\prod_{\substack{j\in \mathbb{I}_m\\j\neq i,i_0}}\sigma(P_{l_j})\\
&\hspace{2cm} \quad\quad\quad +(\sigma(P_{l_{i_0}-3})\sigma(P_{l_{i}-2})-\sigma(P_{l_{i_0}-2})\sigma(P_{l_{i}-3}))\prod_{\substack{j\in \mathbb{I}_m\\j\neq i,i_0}}\sigma(P_{l_j-1})\\
&\hspace{2cm} \quad\quad\quad +(\sigma(P_{l_{i_0}-3})\sigma(P_{l_{i}-3})-\sigma(P_{l_{i_0}-3})\sigma(P_{l_{i}-3}))\prod_{\substack{j\in \mathbb{I}_m\\j\neq i,i_0}}\sigma(P_{l_j-1})\\
&\hspace{2cm}=(\sigma(P_{l_{i_0}-1})\sigma(P_{l_{i}})-\sigma(P_{l_{i_0}})\sigma(P_{l_{i}-1}))\prod_{\substack{j\in \mathbb{I}_m\\j\neq i,i_0}}\sigma(P_{l_j})\\
&\hspace{2cm} \quad\quad\quad +(\sigma(P_{l_{i_0}-3})\sigma(P_{l_{i}-2})-\sigma(P_{l_{i_0}-2})\sigma(P_{l_{i}-3}))\prod_{\substack{j\in \mathbb{I}_m\\j\neq i,i_0}}\sigma(P_{l_j-1})\geq 0.
\end{align*}

If $i\in \mathbb{I}_m$, $l_i\geq 3$ and $i>i_0$ (i.e. $l_i\leq l_{i_0}$), then $l_i$ has to be odd 
and again by Lemma \ref{Lem:Paths} we have
\begin{align*}
&\sigma(U_{i_0}(l_1,\dots,l_m))-\sigma(U_{i}(l_1,\dots,l_m))\\
&\hspace{2cm}=(\sigma(P_{l_{i_0}-1})\sigma(P_{l_{i}})-\sigma(P_{l_{i_0}})\sigma(P_{l_{i}-1}))\prod_{\substack{j\in \mathbb{I}_m\\j\neq i,i_0}}\sigma(P_{l_j})\\
&\hspace{2cm} \quad\quad\quad +(\sigma(P_{l_{i_0}-3})\sigma(P_{l_{i}-2})-\sigma(P_{l_{i_0}-2})\sigma(P_{l_{i}-3}))\prod_{\substack{j\in \mathbb{I}_m\\j\neq i,i_0}}\sigma(P_{l_j-1})\geq 0.
\end{align*}

The proof of the case where all entries (at least 3 in the segment sequence) are odd is similar.
\end{proof}

\begin{proof}[Proof of Theorem \ref{Th:Short_n_Sigma}]
Let $G\in\mathbb{U}(l_1,\dots,l_m)$ be one with the maximum value of $\sigma(G)$. By Lemma \ref{Lem:Gather_Br} we can assume that $G$ is a graph obtained by attaching pendent vertices to vertices of a cycle. Furthermore, all, except possibly one, branching vertices in $G$ must have degree $3$. 
Further let $g$ be the girth of $G$ and let $S(l_{i_1},\dots,l_{i_k})$ denote the starlike graph with segment sequence $(l_{i_1},\dots,l_{i_k})$. First consider the case when $l_1=2$.

If $g\geq l_1+l_2+l_3$, then it is easy to see that $\sigma(G)\leq \sigma(U_1(l_1+l_2+l_3,l_4,\dots,l_{m}))$ and
\begin{align*}
&\sigma(U_{l_1,l_2}(l_4,\dots,l_{m};l_3))-\sigma(U_1(l_1+l_2+l_3,l_4,\dots,l_{m}))\\
= &
\begin{cases}
\prod_{i=4}^{m}\sigma(P_{l_i})(\sigma(S(l_1-1,l_2-1,l_3))-\sigma(P_{l_1+l_2+l_3-1}))\\
+\prod_{i=4}^{m}\sigma(P_{l_i-1})(\sigma(P_{l_3+1})-\sigma(P_{l_1+l_2+l_3-3}))&\text{ if }l_2=2,\\
\\
\prod_{i=3}^{m-1}\sigma(P_{l_i})(\sigma(P_{l_1+l_3})-\sigma(P_{l_1+l_2+l_3-1}))\\
+\prod_{i=3}^{m-1}\sigma(P_{l_i-1})(\sigma(P_{l_3})-\sigma(P_{l_1+l_2+l_3-3}))&\text{ if }l_2=1,\\
\end{cases}\\
\geq & 0,
\end{align*}
with equality only if the two graphs under consideration are isomorphic to each other. Note that replacing a pendent segment of length $2$ by two pendent segments of length $1$ (if necessary) also increases $\sigma(.)$

Now let $l_1+l_2+l_3>g$. Then $g\in\{3,4,5\}$ and we consider different cases:
\begin{itemize}
\item If $g=5$, then we must have $l_2=l_3=2$ and $l_m=1$. Our claim follows from Lemma \ref{Lem:Gather_Br} and the fact that 
\begin{align*} 
&\sigma(U_{4,1}(t_1,\dots,t_s;t)) - \sigma(U_{2,2}(t_1,\dots,t_s,t-1;2))\\
=&
\prod_{i=1}^{s}\sigma(P_{t_i})(\sigma(P_{t+4})-\sigma(P_{t-1})\sigma(S(1,1,2))\\
&\hspace{5cm}+\prod_{i=1}^{s}\sigma(P_{t_i-1})(\sigma(P_{2})\sigma(P_{t})-\sigma{P_{t-2}}\sigma(P_{3}))\\
=&
\prod_{i=1}^{s}\sigma(P_{t_i})(8\sigma(P_{t-1})+5\sigma(P_{t-2})-14\sigma(P_{t-1}))\\
&\hspace{5cm}+\prod_{i=1}^{s}\sigma(P_{t_i-1})(3\sigma(P_{t})-5\sigma{P_{t-2}})<0
\end{align*}
and 
\begin{align*} 
&\sigma(U_{3,2}(t_1,\dots,t_s;t)) - \sigma(U_{2,2}(t_1,\dots,t_s,t-1;2))\\
=&
\prod_{i=1}^{s}\sigma(P_{t_i})(\sigma(S(1,2,t)-\sigma(P_{t-1})\sigma(S(1,1,2))\\
&\hspace{5cm}+\prod_{i=1}^{s}\sigma(P_{t_i-1})(\sigma(P_{t+2})-\sigma{P_{t-2}}\sigma(P_{4}))<0,
\end{align*}
as $P_{t-1}\cup S(1,1,2)$ and $P_{t-2}\cup P_4$ are spanning subgraphs of $S(1,2,t)$ and $P_{t+2}$, respectively.

\item If $g=4$ and $l_2=2$, then the claim follows from Lemma \ref{Lem:Gather_Br} and the fact that
\begin{align*} 
&\sigma(U_{3,1}(t_1,\dots,t_s;t)) - \sigma(U_{2,2}(t_1,\dots,t_s;t))\\
=&
\prod_{i=1}^{s}\sigma(P_{t_i})(\sigma(P_{t+3})-\sigma(S(1,1,t)))
+\prod_{i=1}^{s}\sigma(P_{t_i-1})(2\sigma(P_{t})-\sigma(P_{t+1}))\\
=&
\prod_{i=1}^{s}\sigma(P_{t_i})(\sigma(P_{t+3})-\sigma(S(1,1,t)))
+\prod_{i=1}^{s}\sigma(P_{t_i-1})(2\sigma(P_{t})-\sigma(P_{t+1}))\\
=&-\prod_{i=1}^{s}\sigma(P_{t_i})+\prod_{i=1}^{s}\sigma(P_{t_i-1})<0
\end{align*}
and
\begin{align} 
&\sigma(U_{2,2}(2,t_1,\dots,t_s;1)) - \sigma(U_{2,2}(t_1,\dots,t_s,1;2))\nonumber\\
=&
\prod_{i=1}^{s}\sigma(P_{t_i})(3\sigma(S(1,1,1))-2\sigma(S(1,1,2)))
+\prod_{i=1}^{s}\sigma(P_{t_i-1})(2\sigma(P_{2})-\sigma(P_{3}))\nonumber\\
\label{Eq:g4l2}
=&-\prod_{i=1}^{s}\sigma(P_{t_i})+\prod_{i=1}^{s}\sigma(P_{t_i-1})<0.
\end{align}

\item If $g=4$ and $l_2=1$, then our claim follows from Lemma \ref{Lem:Gather_Br} and the identity
\begin{align*} 
&\sigma(U_1(4,t_1,\dots,t_s)) - \sigma(U_{2,1}(t_1,\dots,t_s;1))\\
&=
\prod_{i=1}^{s}\sigma(P_{t_i})(\sigma(P_{3})-\sigma(P_3))
+\prod_{i=1}^{s}\sigma(P_{t_i-1})(\sigma(P_{1})-\sigma(P_{1}))=0.
\end{align*}
\item If $g=3$ and $l_2=2$, then our claim follows from Lemma \ref{Lem:Gather_Br}, \eqref{Eq:g4l2} and the fact that
\begin{align*} 
&\sigma(U_{2,1}(2,t_1,\dots,t_s;t)) - \sigma(U_{2,2}(t_1,\dots,t_s,1;t))\\
&=
\prod_{i=1}^{s}\sigma(P_{t_i})(3\sigma(P_{t+2})-\sigma(S(1,1,t)))
+\prod_{i=1}^{s}\sigma(P_{t_i-1})(2\sigma(P_{t})-\sigma(P_{t+1}))<0.
\end{align*}

\item If $g=3$ and $l_2=1$, our claim follows from Lemma \ref{Lem:Gather_Br}.

\end{itemize}

If $l_1\neq 2$, then all segments are of length 1. Let $u$ and $v$ be branching vertices of degree $3$ in $G$ and $U^1_m$, respectively. Then, it is easy to see (note that $U^1_m-[v]$ consists of $m-4$ isolated vertices) that
$$
\sigma(U^1_m)=\sigma(U^1_m-v)+\sigma(U^1_m-[v])
\geq \sigma(G-u)+\sigma(G-[u])=\sigma(G)
$$ 
with equality only if $G$ is isomorphic to $U^1_m$. 
\end{proof}

\subsection{Unicyclic graphs with given number of segments}
We first point out the following well known fact that can be directly obtained from Lemma 13 of \cite{Gutman198625} through
$\sigma(P_i\cup P_j)=F_{i+2}F_{j+2}=m(P_{i+1}\cup P_{j+1})$. Here $F_n$'s are the Fibonacci numbers and $m(G)$ is the matching number of $G$.
\begin{lemma}
\label{Lem:Sliding_MissingEdge}
Let $n$ be a positive integer, and write it as 
$n = 4m + h,$ for some $h \in \{1, 2, 3, 4\}$ and for some $m\in\mathbb{N}.$ Then 
\begin{align*}
\sigma(P_{n-2}\cup P_2)  <\sigma(P_{n-4}\cup P_4) < \cdots  <\sigma(P_{n-2m-2l}\cup P_{2m + 2l}) \\
<\sigma(P_{n-2m-1}\cup P_{2m + 1}) < \cdots  <\sigma(P_{n-3}\cup P_3)  <\sigma(P_{n-1}\cup P_1),
\end{align*}

where $l = \lfloor \frac{h-1}{2}\rfloor$.
\end{lemma}

\begin{proof}[Proof of Theorem \ref{theo:segnumSig}]
We will first prove case (i). The idea is very similar to that of the proof of Theorem~\ref{theo:segnum} and we skip some details. 

By Lemma \ref{Lem:Gather_Br}, we may let $G$ be a unicyclic graph obtained from attaching pendent paths to a cycle. Letting $G'$ be obtained from $G$ by moving all pendent paths to one single vertex, we have $\sigma(G')\geq \sigma(G)$. Furthermore, from Lemma \ref{Lem:Sliding}, it suffices now to consider the case where all except possibly one pendant segments of $G'$ are of length $1$. It is easy to see that whenever $G'$ does not have $m$ segments, we can pull an edge from the cycle or from the longest pendent segment, and keep the value of $\sigma(\cdot)$ non-decreasing at all times. In the rest of this proof we compare such a $G'$ with $U=U_2(l_1,\dots,l_m)$.

Denote by $u$ and $v$ the branching vertices in $U=U_2(l_1,\dots,l_m)$ and $G'$, respectively. Let $h_1$ and $h_2$ be the lengths of the cycle and the longest pendent path in $G'$, respectively. Then from our discussion above we have
$h_1\geq 3$ and $h_1+h_2=n-(m-2)$. By Lemma \ref{Lem:Sliding_MissingEdge} we have
\begin{align*}
\sigma(U)
&=\sigma(U-u)+\sigma(U-[u])\\
&=\sigma((m-2)P_1\cup P_3\cup P_{n-m-3})+\sigma(P_1\cup P_{n-m-4})\\
&\geq \sigma((m-2)P_1\cup P_{h_1-1}\cup P_{n-m+1-h_1})+\sigma(P_{h_1-3}\cup P_{n-m+h_1})\\
&=\sigma(G'-v)+\sigma(G'-[v])
=\sigma(G').
\end{align*}
Cases (ii), (iii) and (iv) follow from Theorem \ref{Th:Short_n_Sigma} and the fact that
\begin{align*}
&\sigma(U_{2,2}(1,\dots,1;1))-\sigma(U_1(3,1,\dots,1))\\
&\hspace{4cm}=2^{n-4}\sigma(S(1,1,1))+\sigma(P_2)-2^{n-4}\sigma(P_4)-2^4>0.
\end{align*}
\end{proof}

\section{Concluding remarks}
\label{sec:Con}

In this paper we mainly considered the extremal problems in unicyclic graphs with respect to the number of subtrees and the number of independent sets. Given the rich literature on extremal problems with respect to the number of subtrees in various classes of trees, it is rather surprising that our Theorems~\ref{theo:uni} and \ref{theo:unigir} appear to be the first time any extremal results have been presented for the number of subtrees in unicyclic graphs. The extremal trees that maximize or minimize the number of subtrees among all unicyclic graphs of a given order turned out to be exactly those minimize or maximize the Wiener index among these graphs. This further confirms the well known negative correlation between these two very well studied topological indices. 

We then considered the analogous problems for unicyclic graphs with a given segment sequence, characterizing the extremal graph (that maximize the number of subtrees) in Theorem~\ref{Th:MainSTree}. It is interesting to see examples showing that the analogous statement is not true for the Wiener index. This is perhaps the first pair of non-trivial extremal structures where the negative correlation failed. Based on the information on unicyclic graphs with a given segment sequence and following similar arguments the extremal problem, with respect to the number of subtrees, is also considered for unicyclic graphs with short segments or with a given number of segments, leading to Theorems~\ref{Th:Short_n} and \ref{theo:segnum}.

With respect to the number of independent sets of a graph, known as the Merrifield-Simmons index $\sigma(\cdot)$, similar questions are considered and results analogous to those for the number of subtrees are presented in Theorems~\ref{Th:MainInd}, \ref{Th:Short_n_Sigma} and \ref{theo:segnumSig}. Once again an example is provided, showing that the negative correlation between $\sigma(\cdot)$ and the number of matchings of a graph ($\Z(\cdot)$) fails to hold in unicyclic graphs with a given segment sequence. Noting that this also seems to be the first known case of such examples of extremal structures, it may be interesting to explore what makes the class of unicyclic graphs with a given segment sequence special in such studies.

\section{Acknowledgement} 
The authors would like to thank Prof. Stephan Wagner for his valuable comments that improved this paper.

 \bibliographystyle{abbrv}
 \bibliography{U_segment_sequences}

\begin{thebibliography}{10}

\bibitem{DadahAndriantiana2013}
E.~Andriantiana, S.~Wagner, and H.~Wang.
\newblock Greedy trees, subtrees and antichains.
\newblock {\em Electronic Journal of Combinatorics}, 20(3), 2013.

\bibitem{AWW2015}
E.~O.~D. Andriantiana, S.~Wagner, and H.~Wang.
\newblock Maximum {W}iener index of trees with given segment sequence.
\newblock {\em MATCH Commun. Math. Comput. Chem.}, 75(1):91--104, 2016.

\bibitem{PrevPaper}
E.~O.~D. Andriantiana, S.~Wagner, and H.~Wang.
\newblock Extremal problems for trees with given segment sequence.
\newblock {\em Discrete Applied Mathematics}, 220:20 -- 34, 2017.

\bibitem{du2010minimum}
Z.~Du and B.~Zhou.
\newblock Minimum {W}iener indices of trees and unicyclic graphs of given
  matching number.
\newblock {\em MATCH Commun. Math. Comput. Chem.}, 63(1):101--112, 2010.

\bibitem{Gutman198625}
I.~Gutman and F.~Zhang.
\newblock On the ordering of graphs with respect to their matching numbers.
\newblock {\em Discrete Applied Mathematics}, 15(1):25--33, 1986.

\bibitem{hosoya1971topological}
H.~Hosoya.
\newblock Topological {I}ndex. {A} {N}ewly {P}roposed {Q}uantity
  {C}haracterizing the {T}opological {N}ature of {S}tructural {I}somers of
  {S}aturated {H}ydrocarbons.
\newblock {\em Bull. Chem. Soc. Jpn.}, 44:2332--2339, 1971.

\bibitem{lin2015segment}
H.~Lin and M.~Song.
\newblock On segment sequence and the {W}iener index of trees.
\newblock {\em MATCH Commun. Math. Comput. Chem.}, 75(1):81--89, 2016.

\bibitem{Liu2007183}
H.~Liu and M.~Lu.
\newblock A unified approach to extremal cacti for different indices.
\newblock {\em Match}, 58(1):183--194, 2007.

\bibitem{book:1200029}
L.~Lov\'asz.
\newblock {\em Combinatorial Problems and Exercises}.
\newblock North-holland, Amsterdam, 2 edition, 1993.

\bibitem{merrifield1989topological}
R.~E. Merrifield and H.~E. Simmons.
\newblock {\em Topological {M}ethods in {C}hemistry}.
\newblock Wiley, New York, 1989.

\bibitem{Ren2013768}
C.-R. Ren and J.-S. Shi.
\newblock On the {W}iener index of unicyclic graphs with fixed diameter.
\newblock {\em Huadong Ligong Daxue Xuebao/Journal of East China University of
  Science and Technology}, 39(6):768--772, 2013.

\bibitem{schmuck2012greedy}
N.~S. Schmuck, S.~G. Wagner, and H.~Wang.
\newblock Greedy trees, caterpillars, and {W}iener-type graph invariants.
\newblock {\em MATCH Commun. Math. Comput. Chem.}, 68(1):273--292, 2012.

\bibitem{sze05}
L.~A. Sz{\'e}kely and H.~Wang.
\newblock On subtrees of trees.
\newblock {\em Adv. in Appl. Math.}, 34(1):138--155, 2005.

\bibitem{Tan20151}
S.-W. Tan and Y.~Lin.
\newblock The largest {W}iener index of unicyclic graphs given girth or maximum
  degree.
\newblock {\em Journal of Applied Mathematics and Computing}, 53:343--363,
  2015.

\bibitem{Tan20161}
S.-W. Tan, Q.-L. Wang, and Y.~Lin.
\newblock The {W}iener index of unicyclic graphs given number of pendant
  vertices or cut vertices.
\newblock {\em Journal of Applied Mathematics and Computing}, pages 1--24,
  2016.
\newblock Article in Press.

\bibitem{wagner2007}
S.~Wagner.
\newblock Correlation of graph-theoretical indices.
\newblock {\em SIAM Journal on Discrete Mathematics}, 21:33--46, 2007.

\bibitem{wagner2007extremal}
S.~Wagner.
\newblock Extremal trees with respect to {H}osoya index and
  {M}errifield-{S}immons index.
\newblock {\em MATCH Commun. Math. Comput. Chem.}, 57(1):221--233, 2007.

\bibitem{Wagner2010323}
S.~Wagner and I.~Gutman.
\newblock Maxima and minima of the hosoya index and the merrifield-simmons
  index, a survey of results and techniques.
\newblock {\em Acta Applicandae Mathematicae}, 112(3):323--346, 2010.

\bibitem{wang2008trees}
S.~Wang and X.~Guo.
\newblock Trees with extremal {W}iener indices.
\newblock {\em MATCH Commun. Math. Comput. Chem.}, 60(2):609--622, 2008.

\bibitem{wiener1947structural}
H.~Wiener.
\newblock Structural determination of paraffin boiling points.
\newblock {\em J. Amer. Chem. Soc.}, 69:17--20, 1947.

\bibitem{yu2010}
G.~Yu and L.~Feng.
\newblock On the wiener index of unicyclic graphs with given girth.
\newblock 94, 01 2010.

\bibitem{zhang2008wiener}
X.-D. Zhang, Q.-Y. Xiang, L.-Q. Xu, and R.-Y. Pan.
\newblock The {W}iener index of trees with given degree sequences.
\newblock {\em MATCH Commun. Math. Comput. Chem.}, 60(2):623--644, 2008.

\bibitem{zhang2015}
X.-M. Zhang and X.-D. Zhang.
\newblock The minimal number of subtrees with a given degree sequence.
\newblock {\em Graphs Combin.}, 31:309--318, 2015.

\bibitem{zhang12}
X.-M. Zhang, X.-D. Zhang, D.~Gray, and H.~Wang.
\newblock The number of subtrees of trees with given degree sequence.
\newblock {\em J. Graph Theory}, 73(3):280--295, 2013.

\bibitem{zhao2006fibonacci}
H.~Zhao and X.~Li.
\newblock On the {F}ibonacci numbers of trees.
\newblock {\em Fibonacci Quart.}, 44(1):32--38, 2006.

\end{thebibliography}

\end{document}